\documentclass{amsproc}
\usepackage[margin=1.2in,nomarginpar]{geometry}
\usepackage{amssymb}
\usepackage{amsmath}
\usepackage{mathdots}
\usepackage{amsbsy}
\usepackage{amscd}
\usepackage{amsthm}

\usepackage{url}
\usepackage{xcolor}
\usepackage{amsmath,amssymb,amsthm,amsfonts,longtable}
\usepackage[english]{babel}
\usepackage{tikz-cd}
\usetikzlibrary{cd}
\usetikzlibrary{decorations.markings}
\tikzset{negated/.style={
		decoration={markings,
			mark= at position 0.5 with {
				\node[transform shape] (tempnode) {$\times$};
			}
		},
		postaction={decorate}
	}
}
\usepackage{array}
\usepackage[colorlinks,citecolor=red,urlcolor=blue,bookmarks=false,hypertexnames=true]{hyperref} 
\newtheorem{theorem}{Theorem}
\newtheorem{corollary}[theorem]{Corollary}
\newtheorem{lemma}[theorem]{Lemma}
\newtheorem{proposition}[theorem]{Proposition}

\newtheorem{remark}[theorem]{Remark}
\newtheorem{algorithm}[theorem]{Algorithm}
\newtheorem{example}[theorem]{Example}
\newtheorem{definition}{Definition}[subsection]
\newcommand{\Irr}{\textnormal{Irr}}
\newcommand{\cd}{\textnormal{cd}}
\newcommand{\nl}{\textnormal{nl}}
\newcommand{\lin}{\textnormal{lin}}

\newcommand{\Core}{\textnormal{Core}}
\newcommand{\gal}{\textnormal{Gal}}
\newcommand{\aut}{\textnormal{Aut}}

\title[]{On Matrix Representations of Groups of Order $p^5$ over $\mathbb{Q}$}
\author{Ram Karan Choudhary}
\address{Indian Institute of Technology, Bhubaneswar, Arugul Campus, Jatni, Khurda-752050, India.}
\email{ramkchoudhary1997@gmail.com, rc13@iitbbs.ac.in}
\author{Sunil Kumar Prajapati$^*$}
\address{Indian Institute of Technology, Bhubaneswar, Arugul Campus, Jatni, Khurda-752050, India.}
\email{skprajapati@iitbbs.ac.in}
\thanks{$^{\textbf{*}}$ Corresponding author.}
\begin{document}
	
	\begin{abstract}
		 In this article, we determine all inequivalent irreducible rational matrix representations of groups of order $p^5$, where $p$ is an odd prime. We also derive combinatorial formulations for the Wedderburn decomposition of rational group algebras of these $p$-groups, using results from their rational representations.
	\end{abstract}
	\maketitle

	\section{Introduction}
	The construction of matrix representations of a finite group over a field is a longstanding problem in mathematics. Frobenius initiated the study of representations of finite groups over $\mathbb{C}$, and Schur extended this to subfields of $\mathbb{C}$, particularly $\mathbb{R}$ and $\mathbb{Q}$. Although extensive literature exists on computing matrix representations, methods are typically known only for specific types of representations of particular classes of groups. For example, an algorithm for computing an irreducible complex matrix representation affording the character $\chi$ of $G$ (when $\chi(1) \leq 100$) is described in \cite{Dabbaghian, Dixon} and is implemented as the \textsc{Repsn} package in {\sf GAP} \cite{Gap}. However, in general, computing all inequivalent irreducible matrix representations of a finite group over a field $\mathbb{F}$, including $\mathbb{F} = \mathbb{C}$, remains a challenging and fundamental problem. Recently, rational-valued irreducible complex characters of finite groups have been extensively studied (see \cite{Grittini, Navarro, Navarro1, Navarro2}). This paper focuses on constructing all inequivalent irreducible matrix representations of a $p$-group $G$ of order $p^5$ ($p$ an odd prime) over $\mathbb{Q}$. Studying matrix representations of finite groups over $\mathbb{Q}$ is important for various reasons. For example, a core question in rationality theory concerns the realizability of an $\mathbb{F}$-representation of $G$ over its subfields, such as the realizability of a $\mathbb{C}$-representation over $\mathbb{R}$ or $\mathbb{Q}$. Rational representations can be made integral due to Burnside’s result \cite{Burnside}, which states that for a representation $\rho: G \to GL_n(\mathbb{Q})$, there exists a conjugate representation $\tilde{\rho}$ such that $\tilde{\rho}(g) \in GL_n(\mathbb{Z})$ for all $g \in G$, implying that all matrix entries are integers. Matrix representations of finite groups over $\mathbb{Q}$ frequently arise in various branches of mathematics, similar to those over $\mathbb{C}$, enhancing their significance (see \cite{Bouc, Kletzing, Plesken}).
	
	Throughout this article, we denote a finite group by $G$, the set of all irreducible complex characters of $G$ by $\Irr(G)$, and an odd prime by $p$. For $\chi \in \Irr(G)$, we define
	\begin{equation*}
		\Omega(\chi) = m_{\mathbb{Q}}(\chi) \sum_{\sigma \in \gal(\mathbb{Q}(\chi) / \mathbb{Q})} \chi^{\sigma},
	\end{equation*}
	where $m_{\mathbb{Q}}(\chi)$ is the Schur index of $\chi$ over $\mathbb{Q}$. Note that $\Omega(\chi)$ represents the character of an irreducible $\mathbb{Q}$-representation $\rho$ of $G$. Conversely, if $\rho$ is an irreducible $\mathbb{Q}$-representation of $G$, then there exists $\chi \in \Irr(G)$ such that $\Omega(\chi)$ is the character of $\rho$. In \cite{Ram}, we present an algorithm for constructing irreducible rational matrix representations of $p$-groups. For a $p$-group $G$ and $\chi \in \Irr(G)$, constructing an irreducible rational matrix representation affording the character $\Omega(\chi)$ is equivalent to determining a pair $(H, \psi)$, where $H$ is a subgroup of $G$ and $\psi \in \lin(H)$ satisfies $\psi^G = \chi$ and $\mathbb{Q}(\psi) = \mathbb{Q}(\chi)$ (see Algorithm \ref{algorithm}). This pair $(H, \psi)$ is \emph{required pair} for an irreducible rational matrix representation of $G$ affording $\Omega(\chi)$. Moreover, $(H, \psi)$ also facilitates the computation of an irreducible complex matrix representation of $G$ affording the character $\chi$. In this paper, we identify a required pair for each inequivalent irreducible rational matrix representation of groups of order $p^5$. We have implemented our results in \textsc{Magma} \cite{Magma}, with the corresponding code publicly available in \cite{Magma} via a \texttt{GitHub} repository \cite{Repository}. These implementations enable the explicit construction of irreducible rational matrix representations of such $p$-groups using required pairs.
	
	Further, this article explores the Wedderburn decomposition of the rational group algebras of groups of order $p^5$. The study of such decompositions has garnered significant attention in recent research due to their importance in understanding various algebraic structures (see \cite{Herman, Jes-Rio, Rit-Seh}). Significant work on the Wedderburn decomposition of rational group algebras for various families of groups is documented in \cite{BM14, BGO, Jes-Lea-Paq, Jes-Olt-Rio, Olt07}. In these studies, concepts such as the character value field, Shoda pairs, the set of primitive central idempotents, and numerical representations of cyclotomic algebras have been employed to compute the simple components of the rational algebra. Based on this literature, a package in {\sf GAP} \cite{Gap} named \textsc{Wedderga} has been developed. However, in practice, exact computations remain challenging, particularly for large groups. Combinatorial formulations for the Wedderburn decomposition of the rational group algebras associated with certain classes of finite groups have been provided in \cite{Ram3, Ram2, Ram, PW}. In this article, we extend these results and present combinatorial formulations for the Wedderburn decomposition of the rational group algebras of all $p$-groups of order $p^5$, utilizing results related to the rational matrix representations of these groups, thereby offering an application of these findings.
	
    The article is organized as follows. Section \ref{section:notation} introduces the notation, and Section \ref{sec:preliminaries} presents essential preliminary results. We use James' classification \cite{RJ} of $p$-groups of order $p^5$, which categorizes them into 10 isoclinic families. The number of isomorphism classes of irreducible representations of a finite group $G$ over $\mathbb{Q}$ corresponds to the number of conjugacy classes of cyclic subgroups of $G$ (see \cite[Corollary 1, Chapter 13]{Serre}). Section \ref{section:rational representation} describes all inequivalent irreducible rational matrix representations of these groups, either individually or by combining certain families. Finally, Section \ref{section:applications} provides a combinatorial description for the Wedderburn decomposition of rational group algebras for all $p$-groups of order $p^5$.

	\section{Notation}\label{section:notation}
This section defines the notation, following standard conventions. Throughout the paper, $p$ denotes an odd prime. For a finite group $G$, the following notation is used consistently.

\begin{longtable}{cl}
	$G'$ & the commutator subgroup of $G$\\
	$C_G(H)$ & the centralizer subgroup of $H$ in $G$\\
	$|S|$ & the cardinality of a set $S$\\
	$C_G(H)$ & the centralizer subgroup of $H$ for $H\leq G$\\
	$\Irr(G)$ & the set of irreducible complex characters of $G$\\
	$\lin(G)$ & $\{\chi \in \Irr(G) : \chi(1)=1\}$\\
	$\nl(G)$ & $\{\chi \in \Irr(G) : \chi(1) \neq 1\}$\\
	$\Irr^{(m)}(G)$ & $\{\chi \in \Irr(G) : \chi(1)=m\}$\\
	$\cd(G)$ & $\{ \chi(1) : \chi \in \Irr(G) \}$\\
	$\mathbb{F}(\chi)$ & the field obtained by adjoining the values $\{\chi(g) : g\in G\}$ to the field $\mathbb{F}$, for some $\chi\in \Irr(G)$\\
	$m_\mathbb{Q}(\chi)$ & the Schur index of $\chi \in \Irr(G)$ over $\mathbb{Q}$\\
	$\Omega(\chi)$ & $m_{\mathbb{Q}}(\chi)\sum_{\sigma \in \gal(\mathbb{Q}(\chi) / \mathbb{Q})}^{}\chi^{\sigma}$, for $\chi \in \Irr(G)$\\
	$\ker(\chi)$ & $\{g \in G: \chi(g)=\chi(1)\}$, for $\chi \in \Irr(G)$\\
	$\Irr(G|N)$ & $\{\chi \in \Irr(G) : N \nsubseteq \ker(\chi)\}$, where $N \trianglelefteq G$\\
	$\psi^G$ & the induced character of $\psi$ to $G$, where $\psi$ is a character of $H$ for some $H \leq G$\\
	$\Psi^G$ & the induced representation of $\Psi$ to $G$, where $\Psi$ is a representation of $H$ for some $H \leq G$\\
	$\chi \downarrow_H$ & the restriction of a character $\chi$ of $G$ on $H$, where $H \leq G$\\
	$\mathbb{F}G$ & the group ring (algebra) of $G$ with coefficients in $\mathbb{F}$\\
	$M_{n}(D)$ & a full matrix ring of order $n$ over the skewfield $D$\\
	$Z(B)$ & the center of an algebraic structure $B$\\
	$\phi(n)$ &  Euler's totient function\\
	$\zeta_m$ & an $m$-th primitive root of unity\\
\end{longtable}
	
	\section{Preliminaries}\label{sec:preliminaries}
	In this section, we introduce the key prerequisites for this article. We begin with the definition of a \emph{Camina pair}, a term introduced by Camina in \cite{Camina}.
 \begin{definition}
 	Let $N$ be a normal subgroup of $G$. The pair $(G, N)$ is a Camina pair if $1 < N < G$ and, for every $g \in G \setminus N$, $g$ is conjugate to each element of the coset $gN$.
 \end{definition}
 
\noindent A necessary and sufficient condition for the pair $(G, N)$ to be a Camina pair is that $\chi$ vanishes on $G \setminus N$ for every $\chi \in \Irr(G \mid N)$. It is easy to verify that if $(G, N)$ is a Camina pair, then the chain of inclusions $Z(G) \leq N \leq G'$ holds. In \cite{MLL3}, Lewis first explored groups $G$ for which $(G, Z(G))$ forms a Camina pair and established that such a group must be a $p$-group for some prime $p$. The following lemma highlights a connection between the sets $\Irr(G \mid Z(G))$ and $\Irr(Z(G))$ when $(G, Z(G))$ is a Camina pair.

 \begin{lemma}\label{lemma:Caminacharacter}\textnormal{\cite[Lemma 3.3]{SKP}}
 	Let $(G, Z(G))$ be a Camina pair. Then there exists a bijection between the sets $\Irr(G | Z(G))$ and $\Irr(Z(G)) \setminus \{1_{Z(G)}\}$, where $1_{Z(G)}$ is the trivial character of $Z(G)$. For $1_{Z(G)} \neq \mu \in \Irr(Z(G))$, the corresponding $\chi_\mu \in \nl(G)$ is given by
 	\begin{equation}\label{Caminacharacter}
 		\chi_\mu(g) = 
 		\begin{cases}
 			|G/Z(G)|^{\frac{1}{2}} \mu(g) & \text{if } g \in Z(G), \\
 			0 & \text{otherwise.}
 		\end{cases}
 	\end{equation}
 \end{lemma}
 
 A pair $(G, N)$ is a \emph{generalized Camina pair} if $N$ is a normal subgroup of $G$ and every non-linear irreducible complex character of $G$ vanishes outside $N$ (see \cite{MLL}). A group $G$ is a \emph{VZ-group} if $(G, Z(G))$ forms a generalized Camina pair. For any VZ-group $G$, it holds that $G' \subseteq Z(G)$, $\cd(G) = \{1, |G/Z(G)|^{\frac{1}{2}}\}$ (see \cite{FM}), and by \cite[Lemma 2.4]{MLL2} both $G/Z(G)$ and $G'$ are elementary abelian $p$-groups for some prime $p$. Furthermore, $|\nl(G)|=|Z(G)| - |Z(G)/G'|$, and there exists a one-to-one correspondence between the sets $\nl(G)$ and $\Irr(Z(G) \mid G')$ (see \cite[Subsection 3.1]{SKP1}). For each $\mu \in \Irr(Z(G) \mid G')$, the corresponding character $\chi_\mu \in \nl(G)$ is given by 
 \begin{equation}\label{VZ}
 	\chi_\mu(g) = 
 	\begin{cases}
 		|G/Z(G)|^{\frac{1}{2}} \mu(g) & \text{if } g \in Z(G), \\
 		0 & \text{otherwise.}
 	\end{cases}
 \end{equation}
 
 Next, we describe an \emph{algorithm} for constructing irreducible rational matrix representations of a $p$-group $G$. Let $\chi \in \Irr(G)$, and consider a splitting field $\mathbb{F}$ for $G$. There exists a unique irreducible representation $\rho$ of $G$ over $\mathbb{Q}$ such that $\chi$ appears as an irreducible constituent of $\rho \otimes_{\mathbb{Q}} \mathbb{F}$ with multiplicity $m_{\mathbb{Q}}(\chi)$. For $\psi \in \lin(G)$, an irreducible matrix representation of $G$ over $\mathbb{Q}$ affording the character $\Omega(\psi)$ can be constructed using Lemma \ref{lemma:YamadaLinear}.
 \begin{lemma}\cite[Proposition 1]{Y}\label{lemma:YamadaLinear}
 	Let $\psi \in \lin(G)$ and let $N = \ker(\psi)$ with $n = [G : N]$. Assume $G = \cup_{i = 0}^{n-1} Ny^i$. Then
 	$$\psi(xy^i) = \zeta_n^i, \quad (0 \leq i < n; \, x \in N).$$
 	Suppose $f(X) = X^s - a_{s-1}X^{s-1} - \cdots - a_1X - a_0$ is the irreducible polynomial of $\zeta_n$ over $\mathbb{Q}$, where $s = \phi(n)$. Then $\Psi$ is an irreducible rational matrix representation of $G$ that affords the character $\Omega(\psi)$, and is given by
 	$$\Psi(xy^i) = \left(\begin{array}{ccccc}
 		0 & 1 & 0 & \cdots & 0\\
 		0 & 0 & 1 & \cdots & 0\\
 		\vdots & \vdots & \vdots & \ddots \\
 		0 & 0 & \cdots & 0 & 1\\
 		a_0 & a_1 & \cdots & \cdots & a_{s-1}
 	\end{array}\right)^i, \quad (0 \leq i < n; \, x \in N).$$
 \end{lemma}
 
\noindent Algorithm \ref{algorithm} outlines a method for constructing an irreducible rational matrix representation of a $p$-group $G$ that affords the character $\Omega(\chi)$ for any $\chi \in \Irr(G)$.
 
 \begin{algorithm}\cite[Algorithm 15]{Ram} \label{algorithm}
 	Input: An irreducible complex character $\chi$ of a finite $p$-group $G$, where $p$ is an odd prime.
 	\begin{enumerate}
 		\item Find a pair $(H, \psi)$, where $H \leq G$ and $\psi \in \lin(H)$, such that $\psi^G = \chi$ and $\mathbb{Q}(\psi) = \mathbb{Q}(\chi)$.
 		\item Find an irreducible $\mathbb{Q}$-representation $\Psi$ of $H$ that affords the character $\Omega(\psi)$.
 		\item Induce $\Psi$ to $G$.
 	\end{enumerate}
 	Output: $\Psi^G$, an irreducible $\mathbb{Q}$-representation of $G$ whose character is $\Omega(\chi)$.
 \end{algorithm}
 
 \noindent Observe that constructing an irreducible rational matrix representation of a finite $p$-group $G$ that affords the character $\Omega(\chi)$, where $\chi \in \Irr(G)$, reduces to identifying a \emph{required pair} $(H, \psi)$, as described in Algorithm \ref{algorithm}. In general, such a pair is not unique. Furthermore, a required pair $(H, \psi)$ can also be used to construct an irreducible complex matrix representation of $G$ that affords the character $\chi$.\\
 
We close this section by quoting some well-known results which we also use in the upcoming sections. 
\begin{lemma}\cite[Corollary 10.14]{I}\label{lemma:schurindexpgroup}
	Let $G$ be a $p$-group (an odd prime), and let $\chi \in \Irr(G)$. Then  $m_\mathbb{Q}(\chi) = 1$.
\end{lemma}
We now introduce an equivalence relation on $\Irr(G)$ based on \emph{Galois conjugacy} over $\mathbb{Q}$. Given $\chi, \psi \in \Irr(G)$, they are Galois conjugates over $\mathbb{Q}$ if $\mathbb{Q}(\chi) = \mathbb{Q}(\psi)$ and there exists $\sigma \in \gal(\mathbb{Q}(\chi) / \mathbb{Q})$ such that $\chi^\sigma = \psi$. Each distinct Galois conjugacy class corresponds to a unique irreducible rational representation of $G$. Furthermore, if $G$ is a finite group and $\chi, \psi \in \Irr(G)$ are Galois conjugates over $\mathbb{Q}$, then $\ker(\chi) = \ker(\psi)$.
  \begin{lemma}\cite[Lemma 29]{Ram}
  	Let $G$ is a finite group, and let $\chi, \psi \in \lin(G)$ such that $\ker(\chi) = \ker(\psi)$. Then $\chi$ and $\psi$ are Galois conjugates over $\mathbb{Q}$.
  \end{lemma} 
Finally, for a finite abelian group $G$ and $\chi, \psi \in \Irr(G)$, we say $\chi$ and $\psi$ are equivalent if $\ker(\chi) = \ker(\psi)$.
\begin{lemma}\label{lemma:Ayoub}\cite[Lemma 1]{Ayoub}
	Let $G$ be a finite abelian group, and let $a_d$ be the number of cyclic subgroups of $G$ of order $d$. Then the number of inequivalent characters $\chi$ such that $\mathbb{Q}(\chi)=\mathbb{Q}(\zeta_d)$ is $a_d$.
\end{lemma}

\section{Rational matrix representations} \label{section:rational representation}
In this section, we classify all inequivalent irreducible rational matrix representations of groups of order $p^5$. The non-linear irreducible complex characters of such groups were studied in \cite{SKP}. However, we introduce a technique to obtain these characters when needed. Hall \cite{PH} introduced isoclinism as a generalization of isomorphism for classifying $p$-groups, and James \cite{RJ} later used it to classify $p$-groups up to order $p^6$. We adopt the notation and classification from \cite{RJ}, where groups of order $p^5$ are categorized into 10 isoclinic families, denoted $\Phi_1, \Phi_2, \dots, \Phi_{10}$ (see \cite[Subsection 4.5]{RJ}). $\Phi_1$ consists of all abelian groups of order $p^5$. We now examine the rational matrix representations of groups in the remaining families.
	
	\subsection{Groups belonging to $\Phi_2$ and $\Phi_5$} \label{subsec:Phi2 and Phi5}
	In this subsection, we address the groups belonging to $\Phi_2$ and $\Phi_5$. Note that if $G\in \Phi_2 \cup \Phi_5$, then $G$ is a VZ $p$-group (see \cite[Lemma 5.1]{SKP}). Theorem \ref{thm:reqpairPhi2} provides a detailed description of the rational matrix representations of all groups of order $p^5$ in $\Phi_2$.
	
	\begin{theorem}\label{thm:reqpairPhi2}
		Let $G$ be a group of order $p^5$ such that $G \in \Phi_2$. Then we have the following.
		\begin{enumerate}
			\item Table $\ref{t:1}$ determines all inequivalent irreducible rational matrix representations of $G$ whose kernels do not contain $G'$, where $G \in \Phi_2 \setminus \{\Phi_2(32){a_2}, \Phi_2(221)d\}$.
			\begin{tiny}
				\begin{longtable}[c]{|c|c|c|c|c|}
					\caption{A required pair $(H, \psi_\mu)$ to construct an irreducible rational matrix representation of $G \in \Phi_{2}$ that affords the character $\Omega(\chi_\mu)$, where $\chi_\mu \in \nl(G)$ (as defined in \eqref{VZp^5Phi_2}). \label{t:1}}\\
					\hline
					Group $G$ & $Z(G)$ & $G'$ & $H$ & $\psi_\mu$ \\
					\hline
					\endfirsthead
					\hline
					\multicolumn{5}{|c|}{Continuation of Table $\ref{t:1}$}\\
					\hline
					Group $G$ & $Z(G)$ & $G'$ & $H$ & $\psi_\mu$ \\
					\hline
					\endhead
					\hline
					\endfoot
					\hline
					\endlastfoot
					
					\hline
					\vtop{\hbox{\strut $\Phi_2(311)a = \langle \alpha, \alpha_1, \alpha_2, \gamma : [\alpha_1, \alpha]=\alpha^{p^2}=\alpha_2,$}\hbox{\strut $\alpha_1^p=\alpha_2^p=\gamma^p=1 \rangle$}}  & $\langle \alpha^p, \gamma \rangle$ & $\langle \alpha^{p^2} \rangle$ & $\langle \alpha^p, \alpha_1, \gamma \rangle$ & $\psi_\mu(h) = \begin{cases}
						\mu(\alpha^p)  &\quad \text{ if } h=\alpha^p,\\
						1  &\quad \text{ if } h=\alpha_1,\\
						\mu(\gamma)  &\quad \text{ if } h=\gamma\\
					\end{cases}$\\
					
					\hline \vtop{\hbox{\strut $\Phi_2(221)a = \langle \alpha, \alpha_1, \alpha_2, \gamma : [\alpha_1, \alpha]=\alpha^p=\alpha_2,$ }\hbox{\strut $ \alpha_1^{p^2}=\alpha_2^p=\gamma^p=1 \rangle$}} & $\langle \alpha^p, \alpha_1^p, \gamma \rangle$ & $\langle \alpha^p \rangle$ & \vtop{\hbox{\strut $\langle  \alpha^{-i}\alpha_1, \alpha^p, \gamma \rangle$}\hbox{\strut ($0 \leq i \leq p-1$)}} & $\psi_\mu(h) = \begin{cases}
						1  &\quad \text{ if } h=\alpha^{-i}\alpha_1,\\
						\mu(\alpha^p)  &\quad \text{ if } h=\alpha^p,\\
						\mu(\gamma)  &\quad \text{ if } h=\gamma\\
					\end{cases}$\\
					
					\hline \vtop{\hbox{\strut $\Phi_2(221)b = \langle \alpha, \alpha_1, \alpha_2, \gamma : [\alpha_1, \alpha]=\alpha^p=\alpha_2,$}\hbox{\strut $ \alpha_1^p=\alpha_2^p=\gamma^{p^2}=1 \rangle$}} & $\langle \alpha^p, \gamma \rangle$ & $\langle \alpha^p \rangle$ & $\langle \alpha^p, \alpha_1, \gamma \rangle$ & $\psi_\mu(h) = \begin{cases}
						\mu(\alpha^p)  &\quad \text{ if } h=\alpha^p,\\
						1  &\quad \text{ if } h=\alpha_1,\\
						\mu(\gamma)  &\quad \text{ if } h=\gamma\\
					\end{cases}$\\
					
					\hline\vtop{\hbox{\strut $\Phi_2(2111)a = \langle \alpha, \alpha_1, \alpha_2, \gamma, \delta : [\alpha_1, \alpha]=\alpha^p=\alpha_2,$} \hbox{\strut $ \alpha_1^p=\alpha_2^p=\gamma^p=\delta^p=1 \rangle$ } } &$\langle \alpha^p, \gamma, \delta \rangle$ &$\langle \alpha^p \rangle$ &$\langle \alpha^p, \alpha_1, \gamma, \delta \rangle$ & $\psi_\mu(h) = \begin{cases}
						\mu(\alpha^p)  &\quad \text{ if } h=\alpha^p,\\
						1  &\quad \text{ if } h=\alpha_1,\\
						\mu(\gamma)  &\quad \text{ if } h=\gamma,\\
						\mu(\delta)  &\quad \text{ if } h=\delta\\
					\end{cases}$\\
					
					\hline \vtop{\hbox{\strut $\Phi_2(2111)b = \langle \alpha, \alpha_1, \alpha_2, \gamma, \delta : [\alpha_1, \alpha]=\gamma^p=\alpha_2,$} \hbox{\strut $ \alpha^p=\alpha_1^p=\alpha_2^p =\delta^p=1 \rangle$ } } & $\langle \gamma, \delta \rangle$ & $\langle \gamma^p \rangle$ & $\langle \alpha, \gamma, \delta \rangle$ & $\psi_\mu(h) = \begin{cases}
						1  &\quad \text{ if } h=\alpha,\\
						\mu(\gamma)  &\quad \text{ if } h=\gamma,\\
						\mu(\delta)  &\quad \text{ if } h=\delta\\
					\end{cases}$\\
					
					\hline \vtop{\hbox{\strut $\Phi_2(2111)c = \langle \alpha, \alpha_1, \alpha_2, \gamma: [\alpha_1, \alpha]=\alpha_2,$} \hbox{\strut $ \alpha^{p^2}=\alpha_1^p=\alpha_2^p =\gamma^p=1 \rangle$ } } & $\langle \alpha^p, \alpha_2, \gamma \rangle$ & $\langle \alpha_2 \rangle$ & $\langle \alpha^p, \alpha_1, \alpha_2, \gamma \rangle$ & $\psi_\mu(h) = \begin{cases}
						\mu(\alpha^p)  &\quad \text{ if } h=\alpha^p,\\
						1  &\quad \text{ if } h=\alpha_1,\\
						\mu(\alpha_2)  &\quad \text{ if } h=\alpha_2,\\
						\mu(\gamma)  &\quad \text{ if } h=\gamma\\
					\end{cases}$\\
					
					\hline \vtop{\hbox{\strut $\Phi_2(2111)d = \langle \alpha, \alpha_1, \alpha_2, \gamma: [\alpha_1, \alpha]=\alpha_2,$} \hbox{\strut $ \alpha^p=\alpha_1^p=\alpha_2^p =\gamma^{p^2}=1 \rangle$ } } & $\langle \alpha_2, \gamma \rangle$ & $\langle \alpha_2 \rangle$ & $\langle \alpha_1, \alpha_2, \gamma \rangle$ & $\psi_\mu(h) = \begin{cases}
						1  &\quad \text{ if } h=\alpha_1,\\
						\mu(\alpha_2)  &\quad \text{ if } h=\alpha_2,\\
						\mu(\gamma)  &\quad \text{ if } h=\gamma\\
					\end{cases}$\\
					
					\hline \vtop{\hbox{\strut $\Phi_2(1^5) = \langle \alpha, \alpha_1, \alpha_2, \gamma, \delta: [\alpha_1, \alpha]=\alpha_2,$} \hbox{\strut $ \alpha^p=\alpha_1^p=\alpha_2^p =\gamma^p=\delta^p=1 \rangle$ } } & $\langle \alpha_2, \gamma, \delta \rangle$ & $\langle \alpha_2 \rangle$ & $\langle \alpha_1, \alpha_2, \gamma, \delta \rangle$ & $\psi_\mu(h) = \begin{cases}
						1  &\quad \text{ if } h=\alpha_1,\\
						\mu(\alpha_2)  &\quad \text{ if } h=\alpha_2,\\
						\mu(\gamma)  &\quad \text{ if } h=\gamma,\\
						\mu(\delta)  &\quad \text{ if } h=\delta\\
						
					\end{cases}$\\
					
					\hline \vtop{\hbox{\strut $\Phi_2(41)=\langle \alpha, \alpha_1, \alpha_2 : [\alpha_1, \alpha] =\alpha^{p^3}=\alpha_2,$}\hbox{\strut $\alpha_1^p=\alpha_2^p=1\rangle$ }} & $\langle \alpha^p \rangle$ & $\langle \alpha^{p^3} \rangle$ & $\langle \alpha^p, \alpha_1 \rangle$ & $\psi_\mu(h) = \begin{cases}
						\mu(\alpha^p)  &\quad \text{ if } h=\alpha^p,\\
						1  &\quad \text{ if } h=\alpha_1\\
					\end{cases}$\\
					
					\hline\vtop{\hbox{\strut $\Phi_2(32)a_1=\langle \alpha, \alpha_1, \alpha_2 : [\alpha_1, \alpha] =\alpha^{p^2}=\alpha_2,$}\hbox{\strut $\alpha_1^{p^2}=\alpha_2^p=1\rangle$}}  & $\langle \alpha^p, \alpha_1^p \rangle$ & $\langle \alpha^{p^2} \rangle$ & $\langle \alpha^p, \alpha_1 \rangle$ & $\psi_\mu(h) = \begin{cases}
						\mu(\alpha^p)  &\quad \text{ if } h=\alpha^p,\\
						(\mu(\alpha_1^p))^{\frac{1}{p}}        &\quad \text{ if } h=\alpha_1
					\end{cases}$\\
					
					\hline \vtop{\hbox{\strut $\Phi_2(311)b=\langle \alpha, \alpha_1, \alpha_2, \gamma : [\alpha_1, \alpha] =\gamma^{p^2}=\alpha_2,$}\hbox{\strut $\alpha^p=\alpha_1^p=\alpha_2^p=1\rangle$}} & $\langle \gamma \rangle$ & $\langle \gamma^{p^2} \rangle$ & $\langle \alpha, \gamma \rangle$ & $\psi_\mu(h) = \begin{cases}
						1  &\quad \text{ if } h=\alpha,\\
						\mu(\gamma)  &\quad \text{ if } h=\gamma\\
					\end{cases}$\\
					
					\hline \vtop{\hbox{\strut $\Phi_2(311)c=\langle \alpha, \alpha_1, \alpha_2 : [\alpha_1, \alpha] =\alpha_2,$}\hbox{\strut $\alpha^{p^3}=\alpha_1^p=\alpha_2^p=1\rangle$}}  & $\langle \alpha^p, \alpha_2 \rangle$ & $\langle \alpha_2 \rangle$ & $\langle \alpha^p, \alpha_1, \alpha_2 \rangle$ & $\psi_\mu(h) = \begin{cases}
						\mu(\alpha^p)  &\quad \text{ if } h=\alpha^p,\\
						1  &\quad \text{ if } h=\alpha_1,\\
						\mu(\alpha_2)  &\quad \text{ if } h=\alpha_2\\
					\end{cases}$\\
					
					\hline \vtop{\hbox{\strut $\Phi_2(221)c=\langle \alpha, \alpha_1, \alpha_2, \gamma : [\alpha_1, \alpha] =\gamma^p=\alpha_2,$}\hbox{\strut $\alpha^{p^2}=\alpha_1^p=\alpha_2^p=1\rangle$ }} & $\langle \alpha^p, \gamma \rangle$ & $\langle \gamma^p \rangle$ & $\langle \alpha^p, \alpha_1, \gamma \rangle$ & $\psi_\mu(h) = \begin{cases}
						\mu(\alpha^p)  &\quad \text{ if } h=\alpha^p,\\
						1  &\quad \text{ if } h=\alpha_1,\\
						\mu(\gamma)  &\quad \text{ if } h=\gamma\\
					\end{cases}$\\
					\hline
				\end{longtable}
			\end{tiny}
		\item For $G=\Phi_2(32){a_2}=\langle \alpha, \alpha_1, \alpha_2 : [\alpha_1, \alpha] =\alpha_1^p=\alpha_2, \alpha^{p^3}=\alpha_2^p=1\rangle$, we have $Z(G)=\langle \alpha^p, \alpha_1^p \rangle$ and $G'=\langle \alpha_1^p \rangle$. Suppose $\mu \in \Irr(Z(G) | G')$ and is given by $\mu(z) = \begin{cases}
			\zeta_p  &\quad \text{ if } z=\alpha_1^p,\\
			\zeta_{p^2}^i  &\quad \text{ if } z=\alpha^p\\
		\end{cases}$, where $0\leq i \leq p^2-1$. Then we have the following two cases.
	\begin{enumerate}
		\item {\bf Case ($\mu(\alpha^p)$ is a primitive $p^2$-th root of unity).} In this case, a required pair to construct an irreducible rational matrix representation of $G$ affording the character $\Omega(\chi_\mu)$ is $(H, \psi_\mu)$ such that $H=\langle \alpha^p, \alpha_1 \rangle$ and $$\psi_\mu(h) = \begin{cases}
			\mu(\alpha^p)  &\quad \text{ if } h=\alpha^p,\\
			(\mu(\alpha_1^p))^{\frac{1}{p}}        &\quad \text{ if } h=\alpha_1.
		\end{cases}$$
		
		\item {\bf Case ($\mu(\alpha^p)$ is not a primitive $p^2$-th root of unity).} In this case, $\mu$ can be defined as $\mu(z) = \begin{cases}
		\zeta_p  &\quad \text{ if } z=\alpha_1^p,\\
		\zeta_p^i  &\quad \text{ if } z=\alpha^p\\
	\end{cases}$, where $0\leq i \leq p-1$. Hence, a required pair to construct an irreducible rational matrix representation of $G$ affording the character $\Omega(\chi_\mu)$ is $(H, \psi_\mu)$ such that $H=\langle \alpha\alpha_1^{-i}, \alpha_1^p \rangle$ and $$\psi_\mu(h) = \begin{cases}
	1  &\quad \text{ if } h=\alpha\alpha_1^{-i},\\
	\mu(\alpha_1^p)        &\quad \text{ if } h=\alpha_1^p.
	\end{cases}$$
	\end{enumerate}
	
\item For $G=\Phi_2(221)d=\langle \alpha, \alpha_1, \alpha_2 : [\alpha_1, \alpha] =\alpha_2, \alpha^{p^2}=\alpha_1^{p^2}=\alpha_2^p=1\rangle$, we have $Z(G)=\langle \alpha^p, \alpha_1^p, \alpha_2 \rangle$ and $G'=\langle \alpha_2 \rangle$. Let $\mu \in \Irr(Z(G) | G')$. Then we have the following two cases.
\begin{enumerate}
	\item {\bf Case ($\mu(\alpha_1^p)=1$).} In this case, a required pair to construct an irreducible rational matrix representation of $G$ affording the character $\Omega(\chi_\mu)$ is $(H, \psi_\mu)$ such that $H=\langle \alpha^p, \alpha_1, \alpha_2 \rangle$ and $$\psi_\mu(h) = \begin{cases}
		\mu(\alpha^p)  &\quad \text{ if } h=\alpha^p,\\
		1 & \quad \text{ if } h= \alpha_1,\\
		\mu(\alpha_2)       &\quad \text{ if } h=\alpha_2.
	\end{cases}$$
	
	\item {\bf Case ($\mu(\alpha_1^p)\neq 1$).} In this case, $\mu(\alpha^p) = (\mu(\alpha_1^p))^i$ for some $0\leq i \leq p-1$ and hence a required pair to construct an irreducible rational matrix representation of $G$ affording the character $\Omega(\chi_\mu)$ is $(H, \psi_\mu)$ such that $H=\langle \alpha\alpha_1^{-i}, \alpha_1^p, \alpha_2 \rangle$ and $$\psi_\mu(h) = \begin{cases}
		1  &\quad \text{ if } h=\alpha\alpha_1^{-i},\\
		\mu(\alpha_1^p)        &\quad \text{ if } h=\alpha_1^p,\\
		\mu(\alpha_2)       &\quad \text{ if } h=\alpha_2.
	\end{cases}$$
\end{enumerate}
	  \end{enumerate} 
	\end{theorem}
\begin{proof} 
Observe that, for $G \in \Phi_2$,  
	we have $|G'|=p$, $G' \subset Z(G)$, $|Z(G)|=p^3$ and $\cd(G)=\{1, p\}$ (see \cite[Subection 4.5]{RJ}). Further, from \eqref{VZ}, each of the non-linear irreducible complex characters of $G$ is of the form $\chi_\mu$ and is given by
	\begin{equation}\label{VZp^5Phi_2}
		\chi_\mu(g) = \begin{cases}
			p\mu(g)  &\quad \text{ if } g \in Z(G),\\
			0            &\quad \text{ otherwise }
		\end{cases}
	\end{equation}
	where $\mu \in \Irr(Z(G) | G')$.	
Let $\chi_\mu \in \nl(G)$ (as defined in \eqref{VZp^5Phi_2}). Suppose $(H, \psi_\mu)$ is a required pair to construct an irreducible rational matrix representation of $G$ that affords the character $\Omega(\chi_\mu)$. By \cite[Proposition 21]{Ram}, we have $Z(G) \subset H$ and $\psi_\mu\downarrow_{Z(G)} = \mu$. Furthermore, \cite[Corollary 27]{Ram} ensures that $H$ is abelian. Since $(H, \psi_\mu)$ is a required pair, we have $\mathbb{Q}(\psi_\mu) = \mathbb{Q}(\chi_\mu) = \mathbb{Q}(\mu)$. Thus, by \cite[Lemma 24]{Ram}, we must choose $\psi_\mu \in \lin(H)$ such that $\ker(\psi_\mu) = |G/Z(G)|^{\frac{1}{2}}|\ker(\mu)| = p|\ker(\mu)|$.
	\begin{enumerate}
		\item Consider the group $G = \Phi_2(221)a$. For each $0 \leq i \leq p-1$, suppose $H=\langle  \alpha^{-i}\alpha_1, \alpha^p, \gamma \rangle$ and define $\mu \in \Irr(Z(G) | G')$ as follows:
		$$\mu(z) = \begin{cases}
			\zeta_p & \text{if } z = \alpha^p, \\
			\zeta_p^i & \text{if } z = \alpha_1^p, \\
			\zeta_p^j & \text{if } z = \gamma
		\end{cases}$$
		for some non-negative integer $j$ with $0 \leq j \leq p-1$, and $z \in Z(G)$. Furthermore, we have $(\alpha^{-i}\alpha_1)^p = \alpha^{-ip}\alpha_1^p$. Thus, $\psi_\mu \in \lin(H)$ (as given in Table \ref{t:1}) satisfies
		$$\psi_\mu(\alpha_1^p) = \left( \psi_\mu(\alpha^{-i}\alpha_1) \right)^p \left(\psi_\mu(\alpha^p)\right)^i = \left(\psi_\mu(\alpha^{p})\right)^{i} = \mu(\alpha_1^p).$$
		Observe that $(H, \psi_\mu)$ meets all the conditions to be a required pair.\\
		Similarly, it is a routine verification that all the pairs $(H, \psi_\mu)$ corresponding to the remaining groups in the isoclinic family $\Phi_2$ (as listed in Table \ref{t:1}) also satisfy the criteria for being required pairs. This completes the proof of Theorem \ref{thm:reqpairPhi2}(1).
		
	\end{enumerate}
Theorem \ref{thm:reqpairPhi2}(2) and Theorem \ref{thm:reqpairPhi2}(3) also follow by a routine verification.  This completes the proof of Theorem \ref{thm:reqpairPhi2}.
\end{proof}

Next, we prove Theorem \ref{thm:reqpairPhi5}, which gives a comprehensive description of the rational matrix representations of the groups of order $p^5$ in $\Phi_5$.
	\begin{theorem}\label{thm:reqpairPhi5}
		Let $G$ be a group of order $p^5$ such that $G \in \Phi_5$. Then Table $\ref{t:2}$ determines all inequivalent irreducible rational matrix representations of $G$ whose kernels do not contain $G'$.
		\begin{tiny}
			\begin{longtable}[c]{|c|c|c|c|c|}
				\caption{A required pair $(H, \psi_\mu)$ to construct an irreducible rational matrix representation of $G \in \Phi_{5}$ that affords the character $\Omega(\chi_\mu)$, where $\chi_\mu \in \nl(G)$ (as defined in \eqref{VZp^5Phi_5}). \label{t:2}}\\
				\hline
				Group $G$ & $Z(G)$ & $G'$ & $H$ & $\psi_\mu$ \\
				\hline
				\endfirsthead
				\hline
				\multicolumn{5}{|c|}{Continuation of Table $\ref{t:2}$}\\
				\hline
				Group $G$ & $Z(G)$ & $G'$ & $H$ & $\psi$ \\
				\hline
				\endhead
				\hline
				\endfoot
				\hline
				\endlastfoot
				\hline \vtop{\hbox{\strut $\Phi_5(2111) =\langle \alpha_1, \alpha_2, \alpha_3, \alpha_4, \beta : [\alpha_1, \alpha_2]= [\alpha_3, \alpha_4]=\alpha_1^p=\beta,$} \hbox{\strut $\alpha_2^p=\alpha_3^p=\alpha_4^p=\beta^p=1 \rangle$ } } & $\langle \alpha_1^p \rangle$ & $\langle \alpha_1^p \rangle$ & $\langle \alpha_1^p, \alpha_2, \alpha_3 \rangle$ & $\psi_\mu(h) = \begin{cases}
					\mu(\alpha_1^p)  &\quad \text{ if } h=\alpha_1^p,\\
					1  &\quad \text{ if } h=\alpha_2,\\
					1  &\quad \text{ if } h=\alpha_3\\
				\end{cases}$\\
				
				\hline \vtop{\hbox{\strut $\Phi_5(1^5) =\langle \alpha_1, \alpha_2, \alpha_3, \alpha_4, \beta : [\alpha_1, \alpha_2]= [\alpha_3, \alpha_4]=\beta,$} \hbox{\strut $\alpha_1^p=\alpha_2^p=\alpha_3^p=\alpha_4^p=\beta^p=1 \rangle$ } } & $\langle \beta \rangle$ & $\langle \beta \rangle$ & $\langle \alpha_1, \alpha_3, \beta \rangle$ & $\psi_\mu(h) = \begin{cases}
					1  &\quad \text{ if } h=\alpha_1,\\
					1  &\quad \text{ if } h=\alpha_3,\\
					\mu(\beta)  &\quad \text{ if } h=\beta\\
				\end{cases}$\\
					\hline
			\end{longtable}
		\end{tiny}
	\end{theorem}
\begin{proof}
Observe that, for $G \in \Phi_5$, we have $|G'| = |Z(G)| = p$ and $\cd(G) = \{1, p^2\}$ (see \cite[Subsection 4.5]{RJ}). Additionally, from \eqref{VZ}, each non-linear irreducible complex character $\chi_\mu$ of $G$ is defined by
\begin{equation}\label{VZp^5Phi_5}
	\chi_\mu(g) = \begin{cases}
		p^2\mu(g)  &\quad \text{if } g \in Z(G),\\
		0          &\quad \text{otherwise,}
	\end{cases}
\end{equation}
where $\mu \in \Irr(Z(G) | G')$.
Let $\chi_\mu \in \nl(G)$ as defined in \eqref{VZp^5Phi_5}. Suppose that $(H, \psi_\mu)$ is a required pair to construct an irreducible rational matrix representation of $G$ that affords the character $\Omega(\chi_\mu)$. Following a similar discussion to that in Theorem \ref{thm:reqpairPhi2}, we conclude that $Z(G) \subset H$, $\psi_\mu\downarrow_{Z(G)} = \mu$, and $H$ is abelian. Moreover, from \cite[Lemma 22]{Ram}, it is essential to select $\psi_\mu \in \lin(H)$ such that $|\ker(\psi_\mu)| = p^2 |\ker(\mu)|$. It is routine to verify that all the pairs $(H, \psi_\mu)$ listed in Table \ref{t:2} satisfy the criteria for being required pairs. This completes the proof of Theorem \ref{thm:reqpairPhi5}.
\end{proof}

\subsection{Groups belonging to $\Phi_4$ and $\Phi_6$} \label{subsec:Phi4}
In this subsection, we address the groups in the isoclinic families $\Phi_4$ and $\Phi_6$ together. We begin with Lemma \ref{lemma:characterPhi4}, which provides a method for determining all non-linear irreducible complex characters of groups of order $p^5$ that belong to $\Phi_4$.
\begin{lemma}\label{lemma:characterPhi4}
	Let $G$ be a group of order $p^5$ such that $G \in \Phi_4$. Then we have the following.
	\begin{enumerate}
		\item $\cd(G)=\{1, p\}$.
		\item There exists a bijection between the sets $\{\bar{\chi} \in \nl(G/K) : C_p \cong K < Z(G)\}$ and $\nl(G)$,where $\bar{\chi}$ lifts to $\chi \in \nl(G)$. 
	\end{enumerate}
\end{lemma}
\begin{proof}Suppose $G$ is a group of order $p^5$ such that $G \in \Phi_4$.
	\begin{enumerate}
		\item It follows from the fact that $|Z(G)|=p^2$ and that $\chi(1)^2$ divides $|G/Z(G)|$ for all $\chi \in \nl(G)$ (see \cite[Theorem 20]{Berkovich}).
	
		\item We have $Z(G)= G' \cong C_p \times C_p$ and $\nl(G)=p^3-p$. Assume that $C_p \cong K < Z(G)$. Observe that $G/K$ is a VZ $p$-group of order $p^4$. Thus, $|\nl(G/K)|=p^2-p$. Further, $|\{K < G : C_p \cong K < Z(G)\}|=p+1$. Hence, the result follows.  \qedhere
	\end{enumerate}
\end{proof}

Lemma \ref{lemma:characterPhi6} provides a method for determining all non-linear irreducible complex characters of groups of order $p^5$ that belong to $\Phi_6$.
\begin{lemma}\label{lemma:characterPhi6}
	Let $G$ be a group of order $p^5$ such that $G \in \Phi_6$. Then we have the following.
	\begin{enumerate}
		\item $\cd(G)=\{1, p\}$.
		\item $\nl(G/Z(G))|=p-1$, and there is a bijection between the sets $\{\bar{\chi} \in \nl(G/K | Z(G)/K) : C_p \cong K < Z(G)\}$ and $\nl(G|Z(G))$, where $\bar{\chi}$ lifts to $\chi \in \nl(G)$.
	\end{enumerate}
\end{lemma}
\begin{proof} For $G\in\Phi_6$, we have $|Z(G)|=p^2$.
	\begin{enumerate}
		\item It follows from an argument similar to that in Lemma \ref{lemma:characterPhi4} (1).
		\item We have $C_p \times C_p \cong Z(G)< G' \cong C_p \times C_p \times C_p$ and $\nl(G)=p^3-1$. Assume that $C_p \cong K < Z(G)$. Observe that $G/K$ is a $p$-group of order $p^4$ of nilpotency class $3$. Thus, $|\nl(G/K | Z(G)/K)|= p^2-p$. Further, $|\{C_p \cong K < Z(G)\}|=p+1$ and $\nl(G/Z(G))|=p-1$. Hence, the result follows.  \qedhere
	\end{enumerate}
\end{proof}

Proposition \ref{prop:requiredpairquotientgroup} establishes a connection between the required pairs of a group and those of its quotient groups.
\begin{proposition}\label{prop:requiredpairquotientgroup}
	Let $G$ be a finite $p$-group and $N \trianglelefteq G$. Suppose $\bar{\chi} \in \Irr(G/N)$ and a required pair to construct an irreducible rational matrix representation of $G/N$ affording the character $\Omega(\bar{\chi})$ is $(H/N, \bar{\psi})$. Assume that $\chi$ is the lift of $\bar{\chi}$ to $G$ and $\psi$ is the lift of $\bar{\psi}$ to $H$. Then $(H, \psi)$ is a required pair to construct an irreducible rational matrix representation of $G$ that affords the character $\Omega(\chi)$.
\end{proposition}
\begin{proof}
	Note that $\bar{\psi}^{G/N} = \bar{\chi}$, and hence $\psi^G=\chi$. Furthermore, $\mathbb{Q}(\bar{\chi}) = \mathbb{Q}(\bar{\psi})$ implies that $\mathbb{Q}(\chi)=\mathbb{Q}(\psi)$. This completes the proof of Proposition \ref{prop:requiredpairquotientgroup}.
\end{proof}

\begin{remark}
\textnormal{In \cite{Ram}, we classified the required pairs to construct all inequivalent irreducible rational matrix representations of $p$-groups of order at most $p^4$. Moreover, Lemmas \ref{lemma:characterPhi4} and \ref{lemma:characterPhi6} establish that all non-linear irreducible complex characters of groups of order $p^5$ in $\Phi_4 \cup  \Phi_6$ are lifts of characters from certain quotient groups of order at most $p^4$. Therefore, a required pair for constructing an irreducible rational matrix representation of a group of order $p^5$ in $\Phi_4 \cup \Phi_6$ can be determined using Proposition \ref{prop:requiredpairquotientgroup}.}
\end{remark}
We illustrate the technique discussed above to find a required pair to construct an irreducible rational matrix representation of a group of order $p^5$ that belongs to $\Phi_4 \cup \Phi_6$, in Example \ref{examp:reqpair}.
\begin{example}\label{examp:reqpair}
	\textnormal{Consider the group
		\begin{equation*}
			G = \Phi_4(1^5)= \langle \alpha, \alpha_1, \alpha_2, \beta_1, \beta_2 : [\alpha_i, \alpha]=\beta_i, \alpha^p=\alpha_i^p= \beta_i^p=1 \, \, (i=1, 2)\rangle.
		\end{equation*}
		Here, $Z(G)=G'=\langle \beta_1, \beta_2 \rangle \cong C_p \times C_p$ (see \cite[Subsection 4.5]{RJ}). Suppose $K=\langle \beta_1 \rangle$. Then
		\begin{equation*}
			G/K = \langle \alpha K, \alpha_1 K, \alpha_2 K, \beta_2 K \rangle \cong \Phi_2(1^4).
		\end{equation*} 
		Further, $Z(G/K)=\langle \beta_2 K, \alpha_1 K \rangle$, $(G/K)'=\langle \beta_2 K \rangle$ and $G/K$ is a VZ $p$-group of order $p^4$. Thus, for $\mu \in \Irr(Z(G/K) | (G/K)')$, $\bar{\chi}_\mu \in \nl(G/K)$ can be defined as follows:
		\begin{equation*}
			\bar{\chi}_\mu(gK) = \begin{cases}
				p\mu(gK)  &\quad \text{ if } gK \in Z(G/K),\\
				0            &\quad \text{ otherwise }.
			\end{cases}
		\end{equation*}
		Observe that from \cite[Theorem 49]{Ram}, $(H/K, \bar{\psi}_\mu)$ is a required pair to obtain an irreducible rational matrix representation of $G/K$ affording the character $\Omega(\bar{\chi}_\mu)$, where 
		\begin{equation*}
			H/K = \langle \alpha K, \beta_2 K, \alpha_1 K \rangle \, \, \, \, \text{and} \, \, \, \, \bar{\psi}_\mu(hK) = \begin{cases}
				1  &\quad \text{ if } h=\alpha K,\\
				\mu(\beta_2K)  &\quad \text{ if } hK=\beta_2K,\\
				\mu(\alpha_1K)  &\quad \text{ if } hK=\alpha_1K.\\
			\end{cases}
		\end{equation*}
		Hence, from Lemma \ref{lemma:characterPhi4}, $\chi_\mu \in \nl(G)$, where $\chi_\mu$ is the lift of $\bar{\chi}_\mu$. Therefore, from Proposition \ref{prop:requiredpairquotientgroup}, $(H, \psi_\mu)$ is a required pair to obtain an irreducible rational matrix representation of $G$ which affords the character $\Omega(\chi_\mu)$, where $H = \langle \alpha, \alpha_1, \beta_1, \beta_2 \rangle$ and $\psi_\mu$ is the lift of $\bar{\psi}_\mu$ to $H$.}
\end{example}

\subsection{Groups belonging to $\Phi_3$ and $\Phi_9$}
In this subsection, we address the groups of the isoclinic families $\Phi_3$ and $\Phi_9$ together. We begin with Lemma \ref{lemma:uniqueabelianPhi3Phi9}.
\begin{lemma}\label{lemma:uniqueabelianPhi3Phi9}
	Let $G$ be a group of order $p^5$ such that $G \in \Phi_3 \cup \Phi_9$. Then the centralizer of $G'$, denoted as $C_G(G')$, is the unique abelian subgroup of $G$ of index $p$.  
\end{lemma}
\begin{proof} Let $G$ be a group of order $p^5$ such that $G \in \Phi_9$. According to \cite[Subsection 4.5]{RJ}, it follows that $|Z(G)|=p$, and from \cite[Subsection 4.1]{RJ}, $\cd(G)=\{1, p\}$. Hence, by \cite[Theorem 12.11]{I}, $G$ possesses an abelian subgroup of index $p$.\\
Furthermore, consider $G$ as a group of order $p^5$ such that $G \in \Phi_3$. According to \cite[Subsection 4.1]{RJ}, $G' \cong C_p \times C_p$. Now, let $G$ act on $G'$ by conjugation, yielding a homomorphism
$$T : G \rightarrow \aut(G'),$$
where $T(g) : G' \rightarrow G'$ is an automorphism defined as $T(g)(x)=gxg^{-1}$ for all $x \in G'$. Since $G' \nsubseteq Z(G)$ (as $G$ is of nilpotency class 3) and $\aut(G') \cong GL_2(\mathbb{Z}/p\mathbb{Z})$, it follows that
$$\textnormal{Image}(T) \cong G/\ker(T) \cong C_p.$$
Thus, $\ker(T)$ is a subgroup of $G$ of index $p$.\\
\textbf{Claim :} $\ker(T)$ is an abelian subgroup of $G$.\\
\textbf{Proof of the claim :} Since $G'$ and $Z(G)$ act trivially on $G'$ by conjugation, it follows that $G'\subset \ker(T)$ and $Z(G)\subset \ker(T)$. For $x \in G'$, $z \in Z(G)$ and $h \in \ker(T)$, 
$$(xz)h  = xhz  = h(h^{-1}xh)z = hT(h^{-1})(x)z = h(xz).$$
Thus, $G'Z(G) \subseteq Z(\ker(T))$. Moreover, since $G' \nsubseteq Z(G)$ and $|Z(G)| = p^2$ (see \cite[Subsection 4.5]{RJ}), $G'Z(G)$ forms a subgroup of $Z(\ker(T))$ of order $p^3$. Hence, $\ker(T)$ is an abelian subgroup of $G$. This completes the proof of the claim.\\
Now, suppose $G$ is a non-abelian $p$-group of order $p^5$ in $\Phi_3$ or $\Phi_9$. As $G' \nsubseteq Z(G)$, $C_G(G') \neq G$. Furthermore, let $H$ be an abelian subgroup of $G$ of index $p$. Since $G/H$ is abelian, $G' \subseteq H$. Consequently, $C_G(G') \supseteq H$, implying $C_G(G') = H$. Thus, $H$ is unique. This completes the proof of Lemma \ref{lemma:uniqueabelianPhi3Phi9}.
\end{proof}
Lemma \ref{lemma:characterPhi3Phi9} provides a method to determine all non-linear irreducible complex characters of groups of order $p^5$ that belong to either $\Phi_3$ or $\Phi_9$.
\begin{lemma}\label{lemma:characterPhi3Phi9}
		Let $G$ be a group of order $p^5$ such that $G \in \Phi_3 \cup \Phi_9$. Then we have the following.
	\begin{enumerate}
		\item $\cd(G)=\{1, p\}$.
		\item If $\psi \in \Irr(C_G(G') | G')$, then $\psi^G \in \nl(G)$. Furthermore, for every $\chi \in \nl(G)$, there exists some $\psi \in \Irr(C_G(G') | G')$ such that $\chi = \psi^G$.
	\end{enumerate}
\end{lemma}
\begin{proof} Suppose $G$ is a group of order $p^5$ such that $G \in \Phi_3 \cup \Phi_6$.
	\begin{enumerate}
		\item It follows from \cite[Theorem 6.15]{I}.
		\item	Consider $\psi \in \Irr(C_G(G') | G')$. Suppose that $\psi^G \notin \nl(G)$. This implies that $\psi^G$ is a sum of some linear characters of $G$. This implies that $G' \subseteq \ker(\psi^G) \subset \ker(\psi)$, which is a contradiction. Hence $\psi^G \in \nl(G)$.\\
		Now, let $\psi \in \Irr(C_G(G') | G')$. Then $\psi^G \in \nl(G)$ which implies that the inertia group $I_G(\psi)= C_G(G')$ \cite[Problem 6.1]{I}. Furthermore, ${\psi^{G}}\downarrow_{C_G(G')} = \sum_{i=1}^{p} \psi_{i}$, where $\psi_{i}$'s are conjugates of $\psi$ in $G$ and $p = |G/I_{G}(\psi)|$. Hence, there are $p$ conjugates of $\psi$ and $\psi^G=\psi_i^G\in \nl(G)$ for each $i$. Observe that if $G \in \Phi_3$, then $|G'|=p^2$ and $|\nl(G)|=\frac{|\Irr(C_G(G') | G')|}{p}=p^3-p$. Similarly, if $G \in \Phi_9$, then $|G'|=p^3$ and $|\nl(G)|=\frac{|\Irr(C_G(G') | G')|}{p}=p^3-1$. This completes the proof of Lemma \ref{lemma:characterPhi3Phi9}. \qedhere
	\end{enumerate}
\end{proof}
Now, we proceed by proving certain results that will aid in establishing Theorems \ref{thm:reqpairPhi3} and \ref{thm:reqpairPhi9}. We begin with Lemma \ref{lemma:fieldextension}.
\begin{lemma}\label{lemma:fieldextension}
	Let $G$ be a $p$-group (where $p$ is an odd prime), and let $1 \neq \chi \in \Irr(G)$. Then $\mathbb{Q}(\chi) \neq \mathbb{Q}$.
\end{lemma}
\begin{proof} 
	   	Let $1 \neq \chi \in \lin(G)$. Then $\mathbb{Q}(\chi) \neq \mathbb{Q}$. Further, let $\chi \in \nl(G)$ and define $\bar{\chi} \in \nl(G/\ker(\chi))$ as $\bar{\chi}(g\ker(\chi)) = \chi(g)$. From Clifford's theorem (see \cite[Theorem 6.2]{I}),
	   	$${\bar{\chi}}\downarrow_{Z(G/\ker(\chi))}(x) = \chi(1)\mu(x),$$
	   	where $\mu \in \Irr(Z(G/\ker(\chi)))$ and $x \in Z(G/\ker(\chi))$. Note that $\bar{\chi}$ is a faithful character of $G/\ker(\chi)$. Since $G/\ker(\chi)$ is a $p$-group, $Z(G/\ker(\chi))$ is non-trivial. Hence, $\mu$ is also non-trivial. Therefore, we can conclude that $\mathbb{Q}(\chi) \neq \mathbb{Q}$.
\end{proof}

Lemma \ref{lemma:reqpairdirectproduct} identifies the required pairs associated with the direct product of groups.
\begin{lemma}\label{lemma:reqpairdirectproduct}
	Let $G=G_1\times G_2$, where $G_2$ is abelian. Let $\chi \in \Irr(G)$ such that $\chi = \chi_1\chi_2$, where $\chi_1\in \Irr(G_1)$ and $\chi_2\in\Irr(G_2)$. Suppose $(H_1, \psi_1)$ is a required pair to find an irreducible rational matrix representation of $G_1$ which affords the character $\Omega(\chi_1)$. Then $(H_1\times G_2, \psi_1\chi_2)$ is a required pair to find an irreducible rational matrix representation of $G$ which affords the character $\Omega(\chi)$.
\end{lemma}

\begin{lemma}\label{lemma:centernon-cyclic}
	Let $G$ be a non-abelian $p$-group such that $Z(G)$ is not cyclic, and let $\chi \in \Irr(G)$. Then $\chi$ is the lift of $\bar{\chi} \in \Irr(G/K)$ for some non-trivial subgroup $K$ of $Z(G)$.
\end{lemma}
\begin{proof}
	As $G$ is a $p$-group and $Z(G)$ is not cyclic, $\chi$ is not faithful (see \cite[Theorem 2.32]{I}). Additionally, $\ker(\chi)\cap Z(G)$ is non-trivial, implying that $\chi$ is the lift of $\bar{\chi} \in \Irr(G/K)$, where $K=\ker(\chi)\cap Z(G)$.
\end{proof}

 Theorem \ref{thm:reqpairPhi3} classifies a required pair for an irreducible rational matrix representation of groups of order $p^5$ in $\Phi_3$.
\begin{theorem}\label{thm:reqpairPhi3}
	Let $G$ be a group of order $p^5$ ($p\geq 5$) such that $G \in \Phi_3$. Then we have the following. 
	\begin{enumerate}
		\item For $G \in \Phi_3 \setminus \{\Phi_3(311)b_r~(r=1, \nu), \Phi_3(221)a, \Phi_3(2111)e \}$, consider $\chi \in \nl(G)$ such that $\chi = \psi^G$ for some $\psi \in \Irr(C_G(G') | G')$. Then $(C_G(G'), \psi)$ is a required pair to construct an irreducible rational matrix representation of $G$ that affords the character $\Omega(\chi)$.
		\item For $G \in \{\Phi_3(221)a, \Phi_3(2111)e \}$, there is a bijection between the sets $\{\bar{\chi} \in \nl(G/K | Z(G)/K) : C_p \cong K < Z(G)\}$ and $\nl(G|Z(G))$, where $\bar{\chi}$ lifts to $\chi \in \nl(G)$. Consequently, the required pairs to construct irreducible rational matrix representations of these groups can be determined using Proposition $\ref{prop:requiredpairquotientgroup}$.
		\item For $G=\Phi_3(311)b_r=\langle \alpha, \alpha_1, \alpha_2, \alpha_3~:~  [\alpha_1, \alpha]=\alpha_{2}, [\alpha_2, \alpha]^r=\alpha_1^{p^2}=\alpha_3, \alpha^p=\alpha_2^p=\alpha_3^p=1\rangle$ (where $r=1, \nu$), $C_G(G')= \langle \alpha_1, \alpha_2 \rangle \cong C_{p^3} \times C_p$. Let $\chi \in \nl(G)$. Then we have the following. 
		\begin{enumerate}
			\item If $\chi = \psi^G$ for some $\psi \in \Irr(C_G(G') | G')$ such that $\psi(\alpha_1)=\zeta_{p^3}$, then $(C_G(G'), \psi)$ is a required pair to construct an irreducible rational matrix representation of $G$ that affords the character $\Omega(\chi)$.
			\item There is a bijection between the sets $\nl(G/K)$, where $K=\langle \alpha_3=\alpha_1^{p^2} \rangle <Z(G)$ and $\{\chi \in\nl(G) : \mathbb{Q}(\chi) \neq \mathbb{Q}(\zeta_{p^3})\}$. Consequently, the required pairs to construct irreducible rational matrix representations of these groups that afford the character $\Omega(\chi)$ for some $\chi \in \nl(G)$ such that $\mathbb{Q}(\chi) \neq \mathbb{Q}(\zeta_{p^3})$ can be determined using Proposition $\ref{prop:requiredpairquotientgroup}$.
		\end{enumerate}
	\end{enumerate}
\end{theorem}
\begin{proof}
	\begin{enumerate}
		\item For $G \in \{\Phi_3(2111)a, \Phi_3(2111)b_r ~ ( r=1, \nu), \Phi_3(1^5)\}$, $G$ is the direct product of a non-abelian group of order $p^4$ of nilpotency class $3$ and a cyclic group of order $p$ (see \cite[Subsection 4.5]{RJ}). Hence, from Lemma \ref{lemma:reqpairdirectproduct} and \cite[Theorem 56]{Ram}, it follows that $(C_G(G'), \psi)$ is a required pair to find an irreducible rational matrix representation of $G$ which affords the character $\Omega(\chi)$, where $\chi \in \nl(G)$ such that $\chi = \psi^G$ for some $\psi \in \Irr(C_G(G') | G')$.\\
		Now, let $G \in \{\Phi_3(311)a, \Phi_3(221)b_r ~ (r=1, \nu), \Phi_3(2111)c,  \Phi_3(2111)d\}$. Suppose $\chi = \psi^G \in \nl(G)$ for some $\psi \in \Irr(C_G(G') | G')$. If $G=\Phi_3(2111)d$, then $C_G(G')=\langle \alpha^p, \alpha_1, \alpha_2, \alpha_3 \rangle \cong C_p \times C_p \times C_p \times C_p$ and $\mathbb{Q}(\psi) = \mathbb{Q}(\zeta_p) = \mathbb{Q}(\chi$). If $G=\Phi_3(311)a$, then $C_G(G')=\langle \alpha^p, \alpha_1, \alpha_2 \rangle \cong C_{p^2} \times C_p \times C_p$. Suppose $\mathbb{Q}(\psi) = \mathbb{Q}(\zeta_p)$. Then $\mathbb{Q}(\psi) = \mathbb{Q}(\chi$). Next, if $\mathbb{Q}(\psi) = \mathbb{Q}(\zeta_{p^2})$, then $\psi(\alpha^p)=\zeta_{p^2}$. Assume that $G = \bigcup \alpha^iH$ $(0 \leq i \leq p-1)$. Then we have the following.
		\begin{align*}
			\psi^G(\alpha^p) & = \sum_{i=0}^{p-1}\psi^{\circ}(\alpha^{-i}\alpha^p \alpha^i), ~ \text{where} ~ {\psi^{\circ}}(g) = \begin{cases}
				\psi(g)  &\quad \text{ if } g\in C_G(G'),\\
				0  &\quad \text{ if } g \notin C_G(G')\\
			\end{cases}\\
			& = p \psi(\alpha^p)\\
			& =  p \zeta_{p^2}.
		\end{align*}
		 Hence, $\mathbb{Q}(\psi) = \mathbb{Q}(\zeta_{p^2}) = \mathbb{Q}(\chi$). Now, if $G=\Phi_3(2111)c$, then $C_G(G')=\langle \gamma, \alpha_1, \alpha_2 \rangle \cong C_{p^2} \times C_p \times C_p$. Suppose $\mathbb{Q}(\psi) = \mathbb{Q}(\zeta_p)$. Then $\mathbb{Q}(\psi) = \mathbb{Q}(\chi$). Next, if $\mathbb{Q}(\psi) = \mathbb{Q}(\zeta_{p^2})$, $\psi(\gamma)=\zeta_{p^2}$. Hence, by the similar discussion, we can check that $\mathbb{Q}(\psi) = \mathbb{Q}(\zeta_{p^2}) = \mathbb{Q}(\chi$). Again, if $G=\Phi_3(221)b_1$, $C_G(G')=\langle \alpha_1, \alpha_2, \alpha^p \rangle \cong C_{p^2} \times C_p \times C_p$.  Suppose $\mathbb{Q}(\psi) = \mathbb{Q}(\zeta_p)$. Then $\mathbb{Q}(\psi) = \mathbb{Q}(\chi$). Next, if $\mathbb{Q}(\psi) = \mathbb{Q}(\zeta_{p^2})$, then $\psi(\alpha_1)=\zeta_{p^2}$. Assume that $G = \bigcup \alpha^iC_G(G')$ $(0 \leq i \leq p-1)$. Then we have the following.
		 	\begin{align*}
		 	\psi^G(\alpha_1) & = \sum_{i=0}^{p-1}\psi^{\circ}(\alpha^{-i}\alpha_1\alpha^i), ~ \text{where} ~ {\psi^{\circ}}(g) = \begin{cases}
		 		\psi(g)  &\quad \text{ if } g\in C_G(G'),\\
		 		0  &\quad \text{ if } g \notin C_G(G')\\
		 	\end{cases}\\
		 	& = \psi(\alpha_1) + \psi(\alpha_1\alpha_2)+ \psi(\alpha_1^{1+p}\alpha_2^2)+\psi(\alpha_1^{1+3p}\alpha_2^3)+ \cdots + \psi(\alpha_1^{1+\frac{(p-1)(p-2)}{2}p}\alpha_2^{p-1})\\
		 	& = \psi(\alpha_1) [1 + \psi(\alpha_2)+ \psi(\alpha_1^{p}\alpha_2^2)+\psi(\alpha_1^{3p}\alpha_2^3)+ \cdots + \psi(\alpha_1^{\frac{(p-1)(p-2)}{2}p}\alpha_2^{p-1})]\\
		 	& = \psi(\alpha_1)\sum_{n=1}^{p} \psi(\alpha_1^{\frac{(n-1)(n-2)}{2}p}) \psi(\alpha_2^{n-1}). 
		 \end{align*}
		 It is observed that $\psi(\alpha_1^p)=\zeta_p$ and hence, $\psi(\alpha_2)=\zeta_p^k$ for some $k$ such that $0 \leq k \leq p-1$. Therefore, we have
		 $$\sum_{n=1}^{p} \psi(\alpha_1^{\frac{(n-1)(n-2)}{2}p}) \psi(\alpha_2^{n-1})= \sum_{n=1}^{p} \zeta_p^{\frac{(n-1)(n-2)}{2}} \zeta_p^{k(n-1)}=  \sum_{n=1}^{p} \zeta_p^{\frac{(n-1)(n-2+2k)}{2}}.$$
		 	In the right most sum of the above equality, observe that for $0 \leq k \leq 2$,  $\text{the~}(p-k)\text{-th term}=\text{the~}(3-k)\text{-th term}$, and for $3\leq k \leq p-1$, $\text{the~}(p-k)\text{-th term}=\text{the~}(p+3-k)\text{-th term}$. Hence, the above sum is not equal to zero as there are exactly $p$ terms.  Therefore, $\psi^G(\alpha_1)= \theta\zeta_{p^2}$ for some $0\neq \theta \in \mathbb{Q}(\zeta_p)$. Thus, $\mathbb{Q}(\psi) =\mathbb{Q}(\psi^G) = \mathbb{Q}(\zeta_{p^2})$. Next for $G =\Phi_3(221)b_\nu$, by using the similar discussion, we get the desired result.
		
		\item For $G \in \Phi_3$, we have $G/Z(G) \cong \Phi_2(1^3)$ (see \cite[Subsection 4.1]{RJ}). Therefore, $\nl(G/Z(G))|=p-1$. Moreover, for $G \in \{\Phi_3(221)a, \Phi_3(2111)e \}$, $Z(G)\cong C_p \times C_p$. Hence, from Lemma \ref{lemma:centernon-cyclic}, there is a bijection between the sets $\{\bar{\chi} \in \nl(G/K | Z(G)/K) : C_p \cong K < Z(G)\}$ and $\nl(G|Z(G))$, where $\bar{\chi}$ lifts to $\chi \in \nl(G)$. Therefore, the result follows.
		
		\item \begin{enumerate}
			\item Let $G=\Phi_3(311)b_1$, and let $\chi \in \nl(G)$. Further, suppose $\chi = \psi^G$ for some $\psi \in \Irr(C_G(G') | G')$ such that $\psi(\alpha_1)=\zeta_{p^3}$. Assume that $G = \bigcup \alpha^iC_G(G')$ $(0 \leq i \leq p-1)$. Then we have the following.
			\begin{align*}
				\psi^G(\alpha_1) & = \sum_{i=0}^{p-1}\psi^{\circ}(\alpha^{-i}\alpha_1\alpha^i), ~ \text{where} ~ {\psi^{\circ}}(g) = \begin{cases}
					\psi(g)  &\quad \text{ if } g\in C_G(G'),\\
					0  &\quad \text{ if } g \notin C_G(G')\\
				\end{cases}\\
				& = \psi(\alpha_1) + \psi(\alpha_1\alpha_2)+ \psi(\alpha_1^{1+p^2}\alpha_2^2)+\psi(\alpha_1^{1+3p^2}\alpha_2^3)+ \cdots + \psi(\alpha_1^{1+\frac{(p-1)(p-2)}{2}p^2}\alpha_2^{p-1})\\
				& = \psi(\alpha_1) [1 + \psi(\alpha_2)+ \psi(\alpha_1^{p^2}\alpha_2^2)+\psi(\alpha_1^{3p^2}\alpha_2^3)+ \cdots + \psi(\alpha_1^{\frac{(p-1)(p-2)}{2}p^2}\alpha_2^{p-1})]\\
				& = \psi(\alpha_1)\sum_{n=1}^{p} \psi(\alpha_1^{\frac{(n-1)(n-2)}{2}p^2}) \psi(\alpha_2^{n-1}). 
			\end{align*}
			We have $\psi(\alpha_1^{p^2})=\zeta_p$. This implies that $\psi(\alpha_2)=\zeta_p^k$ for some $k$ such that $0 \leq k \leq p-1$. Therefore,
			$$\sum_{n=1}^{p} \psi(\alpha_1^{\frac{(n-1)(n-2)}{2}p^2}) \psi(\alpha_2^{n-1})= \sum_{n=1}^{p} \zeta_p^{\frac{(n-1)(n-2)}{2}} \zeta_p^{k(n-1)}=  \sum_{n=1}^{p} \zeta_p^{\frac{(n-1)(n-2+2k)}{2}}.$$
				In the right most sum of the above equality, observe that for $0 \leq k \leq 2$, $\text{the~}(p-k)\text{-th term} = \text{the~}(3-k)\text{-th term}$, and for $3\leq k \leq p-1$, $\text{the~}(p-k)\text{-th term} = \text{the~}(p+3-k)\text{-th term}$. Hence, the above sum is not equal to zero as there are exactly $p$ terms. Therefore, $\psi^G(\alpha_1)= \theta\zeta_{p^3}$ for some $0\neq \theta \in \mathbb{Q}(\zeta_p)$. Thus, $\mathbb{Q}(\psi) =\mathbb{Q}(\psi^G) = \mathbb{Q}(\zeta_{p^3})$. Hence, $(C_G(G'), \psi)$ is a required pair to compute an irreducible rational matrix representation of $G$ that affords the character $\Omega(\chi)$. Next, for $G =\Phi_3(311)b_\nu$, by using the similar discussion, we get the desired result.
			\item Observe that $G/K$ is a VZ $p$-group of order $p^4$. Thus, $|\nl(G/K)|=p^2-p$. Furthermore, $|\{\chi \in\nl(G) : \mathbb{Q}(\chi) = \mathbb{Q}(\zeta_{p^3})\}|=p^3-p^2$. Hence, there is a bijection between the sets $\nl(G/K)$, where $K=\langle \alpha_3=\alpha_1^{p^2} \rangle <Z(G)$ and $\{\chi \in\nl(G) : \mathbb{Q}(\chi) \neq \mathbb{Q}(\zeta_{p^3})\}$. Therefore, the result follows.
		\end{enumerate}
	\end{enumerate}
	This completes the proof of Theorem \ref{thm:reqpairPhi3}.
\end{proof}

Theorem \ref{thm:reqpairPhi9} outlines rational matrix representations of groups of order $p^5$ in $\Phi_9$.
\begin{theorem}\label{thm:reqpairPhi9}
	Let $G$ be a group of order $p^5$ ($p\geq 5$) such that $G \in \Phi_9$. Consider $\chi \in \nl(G)$ such that $\chi = \psi^G$ for some $\psi \in \Irr(C_G(G') | G')$. Then $(C_G(G'), \psi)$ is a required pair to construct an irreducible rational matrix representation of $G$ that affords the character $\Omega(\chi)$.
\end{theorem}
\begin{proof}
	If $G = \Phi_9(2111)a \, \, \text{or} \, \, \Phi_9(1^5)$, then $C_G(G') = \langle \alpha_1, \alpha_2, \alpha_3, \alpha_4 \rangle \cong C_p \times C_p \times C_p \times C_p$ (see \cite[Subsection 4.5]{RJ}). Thus, $\mathbb{Q}(\psi)=\mathbb{Q}(\zeta_p)$.  Observe that $\mathbb{Q}(\chi)=\mathbb{Q}(\psi^G) \subseteq \mathbb{Q}(\psi)$. From Lemma \ref{lemma:fieldextension}, we get $\mathbb{Q}(\chi) =\mathbb{Q}(\psi) = \mathbb{Q}(\zeta_p)$. Further, consider 
	\begin{equation*}
		G = \Phi_9(2111)b_r= \langle \alpha, \alpha_1, \dots, \alpha_4 : [\alpha_i, \alpha]=\alpha_{i+1}, \alpha_1^p=\alpha_4^k,  \alpha^p=\alpha_{i+1}^p=1 \, \, (i=1, 2, 3)\rangle,
	\end{equation*}
	where $k=g^r$ for $r= 1, 2, \cdots, (p-1, 3)$ (see \cite[Subsection 4.5]{RJ}). Here, $C_G(G') = \langle \alpha_1, \alpha_2, \alpha_3 \rangle \cong C_{p^2} \times C_p \times C_p $ and $\mathbb{Q}(\psi) = \mathbb{Q}(\zeta_p)$ or $\mathbb{Q}(\zeta_{p^2})$. If $\mathbb{Q}(\psi) = \mathbb{Q}(\zeta_p)$, then from Lemma \ref{lemma:fieldextension}, $\mathbb{Q}(\psi) = \mathbb{Q}(\chi) = \mathbb{Q}(\zeta_p)$. \\
	Now, suppose $\mathbb{Q}(\psi) = \mathbb{Q}(\zeta_{p^2})$. This implies that $\psi(\alpha_1) = \zeta_{p^2}$, $\psi(\alpha_2)$ and $\psi(\alpha_3)$ are some $p$-th roots of unity. Assume that $G = \bigcup \alpha^iC_G(G')$ $(0 \leq i \leq p-1)$. Then
	\begin{align*}
		\psi^G(\alpha_1) & = \sum_{i=0}^{p-1}\psi^{\circ}(\alpha^{-i}\alpha_1\alpha^i), ~ \text{where} ~ {\psi^{\circ}}(g) = \begin{cases}
			\psi(g)  &\quad \text{ if } g\in C_G(G'),\\
			0  &\quad \text{ if } g \notin C_G(G')\\
		\end{cases}\\
		& = \psi(\alpha_1) + \psi(\alpha_1\alpha_2)+ \psi(\alpha_1\alpha_2^2\alpha_3)+\psi(\alpha_1\alpha_2^3\alpha_3^{1+2}\alpha_4)+\psi(\alpha_1\alpha_2^4\alpha_3^{1+2+3}\alpha_4^{1+(1+2)})+ \cdots \\
		& = \psi(\alpha_1) [1 + \psi(\alpha_2)+ \psi(\alpha_2^2)\psi(\alpha_3)+\psi(\alpha_2^3)\psi(\alpha_3^{1+2})\psi(\alpha_4)+\psi(\alpha_2^4)\psi(\alpha_3^{1+2+3})\psi(\alpha_4^{1+(1+2)})+ \cdots]\\
		& = \psi(\alpha_1)\sum_{n=1}^{p}\bigg[ \psi(\alpha_2^{n-1})\psi(\alpha_3^{\frac{(n-1)(n-2)}{2}})\psi(\alpha_4^{\frac{(n-1)(n-2)(n-3)}{6}})\bigg]\\ 
		& = \theta\zeta_{p^2}, ~ \text{for some} ~ 0\neq \theta \in \mathbb{Q}(\zeta_p).
	\end{align*}
	Therefore, $\mathbb{Q}(\psi) =\mathbb{Q}(\psi^G) = \mathbb{Q}(\zeta_{p^2})$. Hence, $(C_G(G'),\psi)$ is a required pair. This completes the proof of Theorem \ref{thm:reqpairPhi9}.
\end{proof}

\subsection{Groups belonging to $\Phi_7$, $\Phi_8$ and $\Phi_{10}$}
In this subsection, we address the groups of the isoclinic families $\Phi_7$, $\Phi_8$ and $\Phi_{10}$ together. Lemma \ref{lemma:characterPhi7Phi8Phi10} delineates the techniques for deriving all non-linear irreducible complex characters of groups belonging to $\Phi_7$, $\Phi_8$ and $\Phi_{10}$.
\begin{lemma}\label{lemma:characterPhi7Phi8Phi10}
	Let $G$ be a non-abelian $p$-group of order $p^5$ such that $G\in \Phi_7\cup  \Phi_8 \cup \Phi_{10}$. Then we have the following.
	\begin{enumerate}
		\item $\cd(G)=\{1, p, p^2\}$.
		\item There is a bijection between the sets $\Irr^{(p)}(G)$ and $\nl(G/Z(G))$. In other words, each non-linear irreducible complex character of degree $p$ of $G$ is the lift of some non-linear irreducible complex character of $G/Z(G)$.
		\item The pair $(G, Z(G))$ is a Camina pair. In fact, there is a bijection between the sets $\Irr^{(p^2)}(G)$ and $\Irr(Z(G)) \setminus \{1_{Z(G)}\}$. Moreover, for $1_{Z(G)} \neq \mu \in \Irr(Z(G))$, the corresponding $\chi_\mu \in \Irr^{(p^2)}(G)$ is defined as follows:
		\begin{equation}\label{Camina}
			\chi_\mu(g) = \begin{cases}
				p^2\mu(g)  &\quad \text{ if } g \in Z(G),\\
				0            &\quad \text{ otherwise. }
			\end{cases}
		\end{equation}
	\end{enumerate}
\end{lemma}
\begin{proof}
	These assertions follow directly from \cite[Lemma 5.5, Lemma 5.6 and Lemma 5.8]{SKP} and Lemma \ref{lemma:Caminacharacter}.
\end{proof}

\begin{remark}
	\textnormal{Let $G$ be a group of order $p^5$ in $\Phi_7 \cup \Phi_8 \cup \Phi_{10}$, and let $\chi \in \Irr^{(p)}(G)$. By Lemma \ref{lemma:characterPhi7Phi8Phi10}, $\chi$ is the lift of some $\bar{\chi} \in \Irr(G/Z(G))$. Moreover, $G/Z(G) \cong \Phi_2(1^4)$ for $G \in \Phi_7$, $G/Z(G) \cong \Phi_2(22)$ for $G \in \Phi_8$, and $G/Z(G) \cong \Phi_3(1^4)$ for $G \in \Phi_{10}$ (see \cite[Subsection 4.1]{RJ}). Thus, a required pair for constructing an irreducible rational matrix representation of $G$ affording the character $\Omega(\chi)$ can be determined using Proposition \ref{prop:requiredpairquotientgroup} along with \cite[Theorems 49 and 58]{Ram}.}
\end{remark}

Next, let $G$ be a $p$-group such that $(G, Z(G))$ forms a Camina pair. Suppose $(H, \psi_{\mu})$ is a required pair for constructing an irreducible rational matrix representation of $G$ affording the character $\Omega(\chi_\mu)$, where $\chi_\mu \in \Irr(G | Z(G))$ is defined in \eqref{Caminacharacter}. Here, we establish results characterizing $(H, \psi_{\mu})$, analogous to those for required pairs in irreducible rational matrix representations of VZ $p$-groups (see \cite{Ram}).  

\begin{lemma} \label{lem:Caminareqpair}
	Let $G$ be a finite $p$-group such that $(G, Z(G))$ forms a Camina pair. Suppose $H$ is a subgroup of $G$ of index $|G/Z(G)|^{\frac{1}{2}}$ and $\psi \in \lin(H)$. Then $\psi^G = \chi_\mu \in \Irr(G|Z(G))$ (as defined in \eqref{Caminacharacter}) if and only if $Z(G) \subseteq H$ and $\psi \downarrow_{Z(G)} = \mu$ with $\mu \in \Irr(Z(G)) \setminus \{1_{Z(G)}\}$.
\end{lemma}
\begin{proof}
	Let $\psi \in \lin(H)$ such that $\psi^G = \chi_\mu$. Let $T$ be a set of right coset representatives of $H$ in $G$. Then for $g\in G$, we have $\psi^G(g) = \sum_{g_i\in T} {{\psi}^{\circ}}(g_igg_i^{-1})$, where ${\psi}^{\circ}$ is defined by  $\psi^\circ(x)=\psi(x)$ if $x \in H$ and  $\psi^{\circ}(x)=0$ if $x\notin H$. Now, for $z\in Z(G)$, we get $\psi^G(z) = |G/Z(G)|^{\frac{1}{2}}{\psi}^\circ(z)$ and since $\psi^G = \chi_\mu$, we obtain ${\psi}^\circ(z) = \mu(z)=\psi(z)$. This implies that $Z(G) \subseteq H$ and $\psi \downarrow_{Z(G)} = \mu$.\\
	Conversely, assume $H$ is a subgroup of $G$ with index $|G/Z(G)|^{\frac{1}{2}}$, $\psi \in \lin(H)$ and $Z(G) \subseteq H$ with $\psi \downarrow_{Z(G)} = \mu$, where $\mu \in \Irr(Z(G)) \setminus \{1_{Z(G)}\}$. We claim that $\psi^G \in \Irr(G|Z(G))$. Suppose to the contrary, that $\psi^G \notin \nl(G|Z(G))$, then $\psi^G$ must be a sum of some irreducible complex character(s) of $G$. This implies that $Z(G) \subseteq \ker(\psi^G)$.  On the other hand, $\psi \downarrow_{Z(G)} = \mu$, where $\mu \in \Irr(Z(G)) \setminus \{1_{Z(G)}\}$. Therefore, $Z(G) \not\subseteq \ker(\psi^G)$ as $\ker(\psi^G) = \Core_G(\ker(\psi))$ (where $\Core_G(\ker(\psi))$ is the normal core of $\ker(\psi)$ in $G$), and $Z(G) \not\subseteq \ker(\psi)$. This is a contradiction. Hence, the claim follows. This completes the proof of Lemma \ref{lem:Caminareqpair}.
\end{proof}

\begin{lemma}\label{lem:fieldofcharacter}
	Let $G$ be a finite $p$-group such that $(G, Z(G))$ forms a Camina pair, and let $\chi_\mu \in \Irr(G|Z(G))$ (as defined in \eqref{Caminacharacter}). Consider a subgroup $H$ of $G$ with index $|G/Z(G)|^{\frac{1}{2}}$ and $\psi_\mu \in \lin(H)$ such that $\psi_\mu^G = \chi_\mu$. Then $\mathbb{Q}(\psi_\mu) = \mathbb{Q}(\chi_\mu)$ if and only if $|\ker(\psi_\mu)/\ker(\mu)|=|G/Z(G)|^{\frac{1}{2}}$.
\end{lemma}
\begin{proof}
	By Lemma \ref{lem:Caminareqpair}, $\psi_{\mu} \downarrow_{Z(G)} = \mu$ with $\mu \in \Irr(Z(G)) \setminus \{1_{Z(G)}\}$ and $\mathbb{Q}(\chi_\mu)=\mathbb{Q}(\mu)$. 
	Observe that
	\begin{align*}
		\mathbb{Q}(\psi_\mu) = \mathbb{Q}(\chi_\mu)=\mathbb{Q}(\mu) & \iff \mathbb{Q}(\zeta_{|H/\ker(\psi_\mu)|}) = \mathbb{Q}(\zeta_{|Z(G)/\ker(\mu)|})\\
		& \iff |H/\ker(\psi_\mu)| = |Z(G)/\ker(\mu)|\\
		& \iff |\ker(\psi_\mu)| = |H/Z(G)| |\ker(\mu)| \\
		& \iff |\ker(\psi_\mu)| = |G/Z(G)|^{\frac{1}{2}} |\ker(\mu)|.
	\end{align*}
This completes the proof of Lemma \ref{lem:fieldofcharacter}.
\end{proof}

\begin{lemma}\label{lemma:cdG2Camina}
	Let $G$ be a finite $p$-group such that $(G, Z(G))$ forms a Camina pair and $|\cd(G)|=2$. Then $G$ is a VZ-group with $G'=Z(G)$.
\end{lemma}
\begin{proof}
	Since $(G, Z(G))$ is a Camina pair, $Z(G) \subseteq G'$. By Lemma \ref{Caminacharacter}, $\cd(G) = \{1, |G/Z(G)|^{\frac{1}{2}}\}$, which implies that $G$ is a VZ-group. Hence, $G'\subseteq Z(G)$. This completes the proof of Lemma \ref{lemma:cdG2Camina}.
\end{proof}

\begin{corollary}\label{cor:reqpairCaminap5}
	Suppose $G$ is a non-abelian $p$-group of order $\leq p^5$ such that $(G, Z(G))$ forms a Camina pair. Let $H$ be a subgroup of $G$ of index $|G/Z(G)|^{\frac{1}{2}}$ with $\psi \in \lin(H)$ such that $\psi^G \in \Irr(G|Z(G))$.  Then $H$ is a normal abelian subgroup of $G$, and $G'$ is contained in $H$.
\end{corollary}
\begin{proof}
	Observe that $\cd(G) \in \{ \{1, p\}, \{1, p^2 \}, \{1, p, p^2\} \}$. First consider $\cd(G) \in \{ \{1, p\}, \{1, p^2 \} \}$. As $(G, Z(G))$ is a Camina pair and $|\cd(G)|=2$, from Lemma \ref{lemma:cdG2Camina}, $G$ is a VZ $p$-group. Hence, by \cite[Corollary 25]{Ram}, the result follows.\\
	Furthermore, if $\cd(G)=\{1, p, p^2\}$, then $|G|=p^5$ and $|Z(G)|=p$. Note that $|G/H|=p^2$, implying $G' \subset H$. Thus, $H$ is a normal subgroup of $G$. Now, suppose for a contradiction that $H$ is not abelian. Then $H'=Z(G)=Z(H)$. Consequently, $Z(G) \subseteq \ker(\psi)$, leading to $\psi \downarrow_{Z(G)} = 1_{Z(G)}$. Hence, by Lemma \ref{lem:Caminareqpair}, $\psi^G \notin \Irr(G|Z(G))$, which is a contradiction. Thus, $H$ must be abelian. This completes the proof of Corollary \ref{cor:reqpairCaminap5}.
\end{proof}

Theorem \ref{thm:reqpairPhi7Phi8Phi10} provides a method to construct irreducible rational matrix representations corresponding to the irreducible complex characters of degree $p^2$ of groups of order $p^5$ in $\Phi_7\cup  \Phi_8 \cup \Phi_{10}$.
\begin{theorem}\label{thm:reqpairPhi7Phi8Phi10}
	Let $G$ be a group of order $p^5$ ($p\geq 5$) such that $G\in \Phi_7\cup  \Phi_8 \cup \Phi_{10}$, and let $\chi_\mu \in \Irr^{(p^2)}(G)$ as defined in \eqref{Camina}. Then we have the following.
	\begin{enumerate}
		\item For each $G \in \Phi_{7}$, Table $\ref{t:3}$ determines an irreducible rational matrix representations of $G$  that affords the character $\Omega(\chi_\mu)$.
		\begin{tiny}
			\begin{longtable}[c]{|c|c|c|c|c|}
				\caption{A required pair $(H, \psi_\mu)$ to construct an irreducible rational matrix representation of $G \in \Phi_{7}$ that affords the character $\Omega(\chi_\mu)$, where $\chi_\mu \in \nl(G)$ (as defined in \eqref{Camina}). \label{t:3}}\\
				\hline
				Group $G$ & $Z(G)$ & $H$ & $\psi_\mu$ \\
				\hline
				\endfirsthead
				\hline
				\multicolumn{5}{|c|}{Continuation of Table $\ref{t:3}$}\\
				\hline
				Group $G$ & $Z(G)$ & $H$ & $\psi_\mu$ \\
				\hline
				\endhead
				\hline
				\endfoot
				\hline
				\endlastfoot
				\hline \vtop{\hbox{\strut $\Phi_7(2111)a=\langle \alpha, \alpha_1, \alpha_2, \alpha_3, \beta : [\alpha_i, \alpha]=\alpha_{i+1}, [\alpha_1, \beta]= \alpha_3=\alpha^p,$} \hbox{\strut $\alpha_1^{(p)}=\alpha_{i+1}^p=\beta^p=1  \, \, (i=1, 2) \rangle $}}  & $\langle \alpha_3 \rangle$ & $\langle \alpha_2, \alpha_3, \beta \rangle$ & $\psi_\mu(h) = \begin{cases}
					1  &\quad \text{ if } h=\alpha_2,\\
					\mu(\alpha_3)  &\quad \text{ if } h=\alpha_3,\\
					1  &\quad \text{ if } h=\beta\\
				\end{cases}$\\
				
				\hline \vtop{\hbox{\strut $\Phi_7(2111)b_r=\langle \alpha, \alpha_1, \alpha_2, \alpha_3, \beta : [\alpha_i, \alpha]=\alpha_{i+1}, [\alpha_1, \beta]^r= \alpha_3^r=\alpha_1^{(p)},$} \hbox{\strut $\alpha^p=\alpha_{i+1}^p=\beta^p=1 \, \, (i=1, 2)\rangle \, \, \text{for} \, \, r=1 \, \, \text{or} \, \, \nu$}}  & $\langle \alpha_3 \rangle$ & $\langle \alpha_2, \alpha_3, \beta \rangle$ & $\psi_\mu(h) = \begin{cases}
					1  &\quad \text{ if } h=\alpha_2,\\
					\mu(\alpha_3)  &\quad \text{ if } h=\alpha_3,\\
					1  &\quad \text{ if } h=\beta\\
				\end{cases}$\\
				
				\hline \vtop{\hbox{\strut $\Phi_7(2111)c=\langle \alpha, \alpha_1, \alpha_2, \alpha_3, \beta : [\alpha_i, \alpha]=\alpha_{i+1}, [\alpha_1, \beta]= \alpha_3=\beta^p,$} \hbox{\strut $\alpha^p=\alpha_1^p=\alpha_{i+1}^p=1  \, \, (i=1, 2) \rangle $, where $p \geq 5$}}  & $\langle \alpha_3 \rangle$ & $\langle \alpha_1, \alpha_2, \alpha_3 \rangle$ & $\psi_\mu(h) = \begin{cases}
					1  &\quad \text{ if } h=\alpha_1,\\
					1  &\quad \text{ if } h=\alpha_2,\\
					\mu(\alpha_3)  &\quad \text{ if } h=\alpha_3\\
				\end{cases}$\\
				
				\hline \vtop{\hbox{\strut $\Phi_7(2111)c=\langle \alpha, \alpha_1, \alpha_2, \alpha_3, \beta : [\alpha_i, \alpha]=\alpha_{i+1}, [\alpha_1, \beta]= \alpha_3=\beta^3,$} \hbox{\strut $\alpha^3=\alpha_1^3\alpha_3=\alpha_{i+1}^p=1  \, \, (i=1, 2) \rangle $}}  & $\langle \alpha_3 \rangle$ & $\langle \beta^{-2}\alpha_1, \alpha_2, \alpha_3 \rangle$ & $\psi_\mu(h) = \begin{cases}
					1  &\quad \text{ if } h=\beta^{-2}\alpha_1,\\
					1  &\quad \text{ if } h=\alpha_2,\\
					\mu(\alpha_3)  &\quad \text{ if } h=\alpha_3\\
				\end{cases}$\\
				
				\hline \vtop{\hbox{\strut $\Phi_7(1^5)=\langle \alpha, \alpha_1, \alpha_2, \alpha_3, \beta : [\alpha_i, \alpha]=\alpha_{i+1}, [\alpha_1, \beta]= \alpha_3,$} \hbox{\strut $\alpha^p=\alpha_1^{(p)}=\alpha_{i+1}^p=\beta^p=1  \, \, (i=1, 2) \rangle $}}  & $\langle \alpha_3 \rangle$ & $\langle \alpha_2, \alpha_3, \beta \rangle$ & $\psi_\mu(h) = \begin{cases}
					1  &\quad \text{ if } h=\alpha_2,\\
					\mu(\alpha_3)  &\quad \text{ if } h=\alpha_3,\\
					1  &\quad \text{ if } h=\beta\\
				\end{cases}$\\
				
				\hline
			\end{longtable}
		\end{tiny}
	\item For $G \in \Phi_{8}$, Table $\ref{t:4}$ determines an irreducible rational matrix representation of $G$ that affords the character $\Omega(\chi_\mu)$.
	\begin{tiny}
		\begin{longtable}[c]{|c|c|c|c|c|}
			\caption{A required pair $(H, \psi_\mu)$ to construct an irreducible rational matrix representation of $G \in \Phi_{8}$ that affords the character $\Omega(\chi_\mu)$, where $\chi_\mu \in \nl(G)$ (as defined in \eqref{Camina}). \label{t:4}}\\
			\hline
			Group $G$ & $Z(G)$ & $H$ & $\psi_\mu$ \\
			\hline
			\endfirsthead
			\hline
			\multicolumn{5}{|c|}{Continuation of Table $\ref{t:4}$}\\
			\hline
			Group $G$ & $Z(G)$ & $H$ & $\psi_\mu$ \\
			\hline
			\endhead
			\hline
			\endfoot
			\hline
			\endlastfoot
			\hline $\Phi_8(32)=\langle \alpha_1, \alpha_2, \beta : [\alpha_1, \alpha_2]= \beta=\alpha_1^p, \beta^{p^2}=\alpha_2^{p^2}=1$  & $\langle \alpha_1^{p^2} \rangle$ & $\langle \alpha_1^{p^2}, \alpha_2 \rangle$ & $\psi_\mu(h) = \begin{cases}
				\mu(\alpha_1^{p^2})  &\quad \text{ if } h=\alpha_1^{p^2},\\
				1  &\quad \text{ if } h=\alpha_2\\
			\end{cases}$\\
			
			\hline
		\end{longtable}
	\end{tiny}
\item For each $G\in\Phi_{10}$, $(G', \psi_\mu)$ is a required pair to compute an irreducible rational matrix representation of $G$ that affords character $\Omega(\chi_\mu)$, where $\psi_\mu \in \Irr(G' | Z(G))$ such that $\psi_\mu \downarrow_{Z(G)} = \mu$. 
	\end{enumerate}
\end{theorem}
\begin{proof}
	Let $G$ be a non-abelian $p$-group of order $p^5$ ($p\geq 5$) such that $G\in \Phi_7\cup  \Phi_8$, and let $\chi_\mu \in \Irr^{(p^2)}(G)$ as defined in \eqref{Camina}. Suppose $(H, \psi_\mu)$ forms a required pair to obtain an irreducible rational matrix representation of $G$ that affords the character $\Omega(\chi_\mu)$. According to Lemma \ref{lem:Caminareqpair}, it follows that $Z(G) \subset H$ and $\psi_\mu\downarrow_{Z(G)} = \mu$. Additionally, from Corollary \ref{cor:reqpairCaminap5}, we know that $H$ is abelian and contains $G'$. Since $(H, \psi_\mu)$ is a required pair, $\mathbb{Q}(\psi_\mu) = \mathbb{Q}(\chi_\mu)$. Therefore, by Lemma \ref{lem:fieldofcharacter}, we must select $\psi_\mu \in \lin(H)$ such that $|\ker(\psi_\mu)| = |G/Z(G)|^{\frac{1}{2}} |\ker(\mu)|$. This implies that $|\ker(\psi_\mu)| = p^2$ because $Z(G)\cong C_p$. It is routine to check that all the pairs $(H, \psi_\mu)$ mentioned in Table \ref{t:3} and Table \ref{t:4} satisfy the conditions of being required pairs.\\
	 Now, let $G\in\Phi_{10}$. Then $|G'|=p^3$ (see \cite[Subsection 4.5]{RJ}). Hence from Corollary \ref{cor:reqpairCaminap5}, $(G', \psi_\mu)$ is a required pair to compute an irreducible rational matrix representation of $G$ whose character is $\Omega(\chi_\mu)$, where $\psi_\mu \in \Irr(G' | Z(G))$ such that $\psi_\mu \downarrow_{Z(G)} = \mu$. This completes the proof of Theorem \ref{thm:reqpairPhi7Phi8Phi10}.
\end{proof}

\begin{remark}
	\textnormal{Note that a required pair for constructing an irreducible rational matrix representation is generally not unique (see \cite[Remark 46]{Ram}). However, for $G \in \Phi_{10}$ and $\chi_\mu \in \Irr^{(p^2)}(G)$ as defined in \eqref{Camina}, the pair $(G', \psi_\mu)$ is uniquely determined for constructing an irreducible rational matrix representation of $G$ affording the character $\Omega(\chi_\mu)$, where $\psi_\mu \in \Irr(G' | Z(G))$ satisfies $\psi_\mu \downarrow_{Z(G)} = \mu$.}
\end{remark}

\subsection{Examples} Here, we provide a couple of examples to demonstrate the construction of irreducible rational matrix representations of groups of order $p^5$ using required pairs.
\begin{example}
	\textnormal{For $p \geq 5$, consider 
		$$G= \Phi_3(1^5) = \langle \alpha, \alpha_1, \alpha_2, \alpha_3, \alpha_4 : [\alpha_1, \alpha]=\alpha_2, [\alpha_2, \alpha]=\alpha_3, \alpha^p=\alpha_1^p=\alpha_2^p=\alpha_3^p=\alpha_4^p=1 \rangle.$$
		We have $ G' = \langle \alpha_2, \alpha_3 \rangle \cong C_p \times C_p$ and $C_G(G') = \langle \alpha_1, \alpha_2, \alpha_3, \alpha_4 \rangle \cong C_p \times C_p \times C_p \times C_p$. Further, let $\psi \in \Irr(C_G(G') | G')$ such that 
		$$\psi(h) = \begin{cases}
			1  &\quad \text{ if } h=\alpha_1,\\
			\zeta_p  &\quad \text{ if } h=\alpha_2,\\
			1  &\quad \text{ if } h=\alpha_3,\\
			1  &\quad \text{ if } h=\alpha_4.\\
		\end{cases}$$ 
		Suppose $\rho$ is an irreducible rational matrix representation of $C_G(G')$ that affords the character $\Omega(\psi)$. From  Lemma \ref{lemma:YamadaLinear}, we have $\rho(1)=p-1$, and the explicit form of $\rho $ is given by
		$$\rho(\alpha_1)= \rho(\alpha_3)=\rho(\alpha_4) = \left(\begin{array}{ccccc}
			1 & 0 & 0 & \cdots & 0\\
			0 & 1 & 0 & \cdots & 0\\
			\vdots & \vdots & \vdots & \ddots \\
			0 & 0 & \cdots & 1 & 0\\
			0 & 0 & \cdots & \cdots & 1
		\end{array}\right) = I, \text{ and }
		~ \rho(\alpha_2)=\left(\begin{array}{ccccc}
			0 & 1 & 0 & \cdots & 0\\
			0 & 0 & 1 & \cdots & 0\\
			\vdots & \vdots & \vdots & \ddots \\
			0 & 0 & \cdots & 0 & 1\\
			-1 & -1 & \cdots & \cdots & -1
		\end{array}\right).$$ 
		Set $\rho(\alpha_2) = P$, and let $O$ denote the zero matrix of order $(p-1)$. Next, let $\chi \in \nl(G)$ such that $\chi = \psi^G$. Then $(C_G(G'), \psi)$ is a required pair to obtain an irreducible rational matrix representation of $G$ that affords the character $\Omega(\chi)$ (see Theorem \ref{thm:reqpairPhi3}). Hence, $\rho^G$ is an irreducible rational matrix representation of degree $p^2-p$ of $G$ that affords the character $\Omega(\chi)$. The explicit form of $\rho^G$ is given by
		\begin{align*}
			& \rho^G(\alpha)=\left(\begin{array}{ccccc}
				O & O & O & \cdots & I\\
				I & O & O & \cdots & O\\
				O & I & O & \cdots & O\\
				\vdots & \vdots & \vdots & \ddots \\
				O & O & \cdots & I & 0
			\end{array}\right),
			~ \rho^G(\alpha_1)=\left(\begin{array}{ccccc}
				I & O & O & \cdots & O\\
				O & P & O & \cdots & O\\
				O & O & P^2 & \cdots & O\\
				\vdots & \vdots & \vdots & \ddots \\
				O & O & \cdots & \cdots & P^{p-1}
			\end{array}\right),\\
			&\rho^G(\alpha_2)=\left(\begin{array}{ccccc}
				P & O & O & \cdots & O\\
				O & P & O & \cdots & O\\
				O & O & P & \cdots & O\\
				\vdots & \vdots & \vdots & \ddots \\
				O & O & \cdots & O & P
			\end{array}\right), \text{ and }
			~ \rho^G(\alpha_3)= \left(\begin{array}{ccccc}
				I & O & O & \cdots & O\\
				O & I & O & \cdots & O\\
				O & O & I & \cdots & O\\
				\vdots & \vdots & \vdots & \ddots \\
				O & O & \cdots & \cdots & I
			\end{array}\right)= \rho^G(\alpha_4).
	\end{align*}}
\end{example}

\begin{example}
	\textnormal{Consider 
		$$G= \Phi_8(32) = \langle \alpha_1, \alpha_2, \beta : [\alpha_1, \alpha_2]=\beta=\alpha_1^p, \beta^{p^2}=\alpha_2^{p^2} \rangle.$$
		We have $ G' = \langle \alpha_1^p \rangle \cong C_{p^2}$ and $Z(G) = \langle \alpha_1^{p^2} \rangle \cong C_p$. From Lemma \ref{lemma:characterPhi7Phi8Phi10}, the pair $(G, Z(G))$ is a Camina pair. Let $\chi_\mu \in \Irr^{(p^2)}(G)$ as defined in \eqref{Camina}, where $\mu \in \Irr(Z(G))\setminus \{1_{Z(G)}\}$ such that $\mu(\alpha_1^{p^2})=\zeta_p$.  From Theorem \ref{thm:reqpairPhi7Phi8Phi10}, $(H, \psi_\mu)$ is a required pair to obtain an irreducible rational matrix representation of $G$ affording the character $\Omega(\chi_\mu)$, where $H = \langle \alpha_1^{p^2}, \alpha_2 \rangle$ and $\psi_\mu \in \lin(H)$ such that $\psi_\mu(\alpha_1^{p^2}) = \mu(\alpha_1^{p^2})=\zeta_p$ and $\psi_\mu(\alpha_2) = 1$. Let $\Psi_\mu$ denote an irreducible rational matrix representation of $H$ that affords the character $\Omega(\psi_\mu)$. Then $\Psi_\mu(1)=p-1$, and the explicit form of $\Psi_\mu$ is given by
		$$\Psi_\mu(\alpha_1^{p^2})=\left(\begin{array}{ccccc}
			0 & 1 & 0 & \cdots & 0\\
			0 & 0 & 1 & \cdots & 0\\
			\vdots & \vdots & \vdots & \ddots \\
			0 & 0 & \cdots & 0 & 1\\
			-1 & -1 & \cdots & \cdots & -1
		\end{array}\right), \text{ and }
		~ \Psi_\mu(\alpha_2)= \left(\begin{array}{ccccc}
			1 & 0 & 0 & \cdots & 0\\
			0 & 1 & 0 & \cdots & 0\\
			\vdots & \vdots & \vdots & \ddots \\
			0 & 0 & \cdots & 1 & 0\\
			0 & 0 & \cdots & \cdots & 1
		\end{array}\right) = I$$
		(see Lemma \ref{lemma:YamadaLinear}). 
		Next, set $\Psi_\mu(\alpha_1^{p^2}) = P$, and let $O$ denote the zero matrix of order $(p-1)$. Then an irreducible rational matrix representation $\Psi^G$ of degree $p^3-p^2$ of $G$ affording the character $\Omega(\chi_\mu)$ is given by
		\[\Psi_\mu^G(\alpha_1)=\left(\begin{array}{ccccc}
			O & O & O & \cdots & P\\
			I & O & O & \cdots & O\\
			O & I & O & \cdots & O\\
			\vdots & \vdots & \vdots & \ddots \\
			O & O & \cdots & I & 0
		\end{array}\right), \text{ and }
		~ \Psi_\mu^G(\alpha_2)=\left(\begin{array}{ccccc}
			I & O & O & \cdots & O\\
			O & P & O & \cdots & O\\
			O & O & P^2 & \cdots & O\\
			\vdots & \vdots & \vdots & \ddots \\
			O & O & \cdots & \cdots & P^{p^2-1}
		\end{array}\right).\]
		Moreover, $\Psi^G$ is the unique faithful irreducible rational representation of $G$.}
\end{example}

\section{Rational group algebras} \label{section:applications}
In this section, we provide a combinatorial description for the Wedderburn decomposition of rational group algebras of groups of order $p^5$, using results on their rational representations. The Artin-Wedderburn theorem states that a semisimple ring is a direct sum of matrix rings over division rings. By the Brauer-Witt theorem, the Wedderburn components of a rational group algebra are Brauer equivalent to cyclotomic algebras (see \cite{Yam}). Perlis and Walker \cite{PW} studied the rational group algebras of finite abelian groups.

\begin{theorem}[Perlis-Walker theorem]\label{Perlis-walker}
	Let $G$ be a finite abelian group of exponent $m$. Then the Wedderburn decomposition of $\mathbb{Q}G$ is given by
	$$\mathbb{Q}G \cong \bigoplus_{d|m} a_d \mathbb{Q}(\zeta_d),$$
	where $a_d$ denotes the number of cyclic subgroups of $G$ of order $d$.
\end{theorem} 

In \cite{Ram}, we formulate the computation of the Wedderburn decomposition of rational group algebra of a VZ $p$-group $G$ solely based on computing the number of cyclic subgroups of $G/G'$, $Z(G)$ and $Z(G)/G'$, which is analogous to the combinatorial formula given by Perlis and Walker for the Wedderburn decomposition of rational group algebras of finite abelian groups. 

\begin{theorem}\cite[Theorem 1]{Ram}\label{lemma:Wedderburn VZ}
	Let $G$ be a finite VZ $p$-group, where $p$ is an odd prime. Let $m_1$, $m_2$ and $m_3$ denote the exponents of $G/G'$, $Z(G)$ and $Z(G)/G'$, respectively. Then the Wedderburn decomposition of $\mathbb{Q}G$ is given by
	$$\mathbb{Q}G \cong \bigoplus_{d_1|m_1}a_{d_1}\mathbb{Q}(\zeta_{d_1})\bigoplus_{d_2\mid m_2, d_2 \nmid m_3} a_{d_2}M_{|G/Z(G)|^{\frac{1}{2}}}(\mathbb{Q}(\zeta_{d_2}))\bigoplus_{d_2|m_2, d_2|m_3}(a_{d_2}-a_{d_2}')M_{|G/Z(G)|^{\frac{1}{2}}}(\mathbb{Q}(\zeta_{d_2})),$$
	where $a_{d_1}$, $a_{d_2}$ and $a_{d_2}'$ are the number of cyclic subgroups of $G/G'$ of order $d_1$, the number of cyclic subgroups of $Z(G)$ of order $d_2$ and the number of cyclic subgroups of $Z(G)/G'$ of order $d_2$, respectively.
\end{theorem}

Note that groups of order $p^5$ in $\Phi_1$ are abelian, so Theorem \ref{Perlis-walker} determines the Wedderburn decomposition of their rational group algebras. Similarly, groups of order $p^5$ in $\Phi_2 \cup \Phi_5$ are VZ $p$-groups, and the Wedderburn decomposition of their rational group algebras follows from Theorem \ref{lemma:Wedderburn VZ}. We now consider the Wedderburn decomposition of rational group algebras for groups of order $p^5$ in $\Phi_4$.

\begin{theorem}\label{thm:rationalgroupalgebraPhi4}
	Let $G$ be a group of order $p^5$ such that $G \in \Phi_4$, and let $C_p \cong K < Z(G)$. Let $m_K$ and $m_K'$ denote the exponents of $Z(G/K)$ and $Z(G/K)/(G/K)'$, respectively. Then the Wedderburn decomposition of $\mathbb{Q}G$ is as follows:
	\begin{align*}
		\mathbb{Q}G \cong & \bigoplus \mathbb{Q}(G/G') \bigoplus_{C_p \cong K < Z(G)}~ \bigoplus_{\substack{d_K|m_K\\ d_K \nmid m_K'}} a_{d_K}M_p(\mathbb{Q}(\zeta_{d_K}))\\
		& \bigoplus_{C_p \cong K < Z(G)} ~\bigoplus_{\substack{d_K|m_K\\ d_K|m_K'}}(a_{d_K}-a_{d_K}')M_p(\mathbb{Q}(\zeta_{d_K})),
	\end{align*}
	where $a_{d_K}$ and $a_{d_K}'$ are the number of cyclic subgroups of order $d_K$ of $Z(G/K)$ and $Z(G/K)/(G/K)'$, respectively.
\end{theorem}
\begin{proof}
	Let $G$ be a group of order $p^5$ such that $G \in \Phi_4$. For $\chi \in \lin(G)$, there exists $\bar{\chi} \in \Irr(G/G')$ such that $\bar{\chi}$ lifts to $\chi$. Moreover, there is a bijection between $\lin(G)$ and $\Irr(G/G')$. Hence, the simple components of the Wedderburn decomposition of $\mathbb{Q}G$ corresponding to all irreducible $\mathbb{Q}$-representations of $G$ whose kernels contain $G'$ are isomorphic to $\mathbb{Q}(G/G')$.\\
	Further, $Z(G)= G' \cong C_p \times C_p$ (see \cite[Subsection 4.5]{RJ}). Suppose $C_p \cong K < Z(G)$. Observe that $(G/K)' \cong C_p$. Thus, $|Z(G/K)|= p^2$ and $G/K$ is a VZ $p$-group of order $p^4$. From Lemma \ref{lemma:characterPhi4}, every non-linear irreducible complex character of  $G$ is the lift of some non-linear irreducible complex character of $G/K$ for some $K$. More precisely, there exists a bijection between $\{\bar{\chi} \in \nl(G/K) : C_p \cong K < Z(G)\}$ and $\nl(G)$.\\
	Next, fix a	subgroup $K$ such that $C_p \cong K < Z(G)$. Observe that there is a bijection between $\nl(G/K)$ and $\{\chi \in \nl(G) : K\subseteq \ker(\chi)\}$. Hence, from Theorem \ref{lemma:Wedderburn VZ}, the simple components of the Wedderburn decomposition of $\mathbb{Q}G$ corresponding to all irreducible $\mathbb{Q}$-representations of $G$ whose kernels contain $K$ contribute
	$$\bigoplus_{\substack{d_K|m_K\\ d_K \nmid m_K'}} a_{d_K}M_p(\mathbb{Q}(\zeta_{d_K})) \bigoplus_{\substack{d_K|m_K\\ d_K|m_K'}}(a_{d_K}-a_{d_K}')M_p(\mathbb{Q}(\zeta_{d_K}))$$
	in $\mathbb{Q}G$, where $m_K$ and $m_K'$ are the exponents of $Z(G/K)$ and $Z(G/K)/(G/K)'$, and $a_{d_K}$ and $a_{d_K}'$ are the number of cyclic subgroups of order $d_K$ of $Z(G/K)$ and $Z(G/K)/(G/K)'$, respectively. Therefore, the proof follows.
\end{proof}

 Theorem \ref{thm:rationalgroupalgebraPhi6} gives a combinatorial formulation for the Wedderburn decomposition of  rational group algebras associated with the groups of order $p^5$ belonging to $\Phi_6$.
\begin{theorem}\label{thm:rationalgroupalgebraPhi6}
	Let $G$ be a group of order $p^5$ such that $G \in \Phi_6$, and let $C_p \cong K < Z(G)$. Let $m_K$ and $m_K'$ denote the exponents of $C_{G/K}((G/K)')$ and $C_{G/K}((G/K)')/(Z(G)/K)$, respectively. Then the Wedderburn decomposition of $\mathbb{Q}G$ is as follows:
	\begin{align*}
		\mathbb{Q}G \cong & \bigoplus \mathbb{Q}(G/G') \bigoplus M_p(\mathbb{Q}(\zeta_p))~ \bigoplus_{C_p \cong K < Z(G)}~\bigoplus_{\substack{d_K|m_K \\ d_K \nmid m_K'}} \frac{a_{d_K}}{p}M_p(\mathbb{Q}(\zeta_{d_K}))\\
		& \bigoplus_{C_p \cong K < Z(G)}~\bigoplus_{\substack{d_K|m_K \\ d_K|m_K'}}\frac{a_{d_K}-a_{d_K}'}{p}M_p(\mathbb{Q}(\zeta_{d_K})),
	\end{align*}
	where $a_{d_K}$ and $a_{d_K}'$ are the number of cyclic subgroups of order $d_K$ of the centralizer $C_{G/K}((G/K)')$ and the quotient group $C_{G/K}((G/K)')/(Z(G)/K)$, respectively.
\end{theorem}
\begin{proof}
		Let $G$ be a group of order $p^5$ such that $G \in \Phi_6$. By using an analogue of the argument given in the proof of Theorem \ref{thm:rationalgroupalgebraPhi4}, we conclude that the simple components of the Wedderburn decomposition of $\mathbb{Q}G$ corresponding to all irreducible $\mathbb{Q}$-representations of $G$ whose kernels contain $G'$ are isomorphic to $\mathbb{Q}(G/G')$.\\
	    Further, $C_p \times C_p \cong Z(G) < G' \cong C_p \times C_p \times C_p$ and $G/Z(G) \cong \Phi_2(1^3)$ (see \cite[Subsection 4.1]{RJ}).  Hence,  
	    $$\bigoplus M_{p}(\mathbb{Q}(\zeta_p))$$
       is the only simple component of the Wedderburn decomposition of $\mathbb{Q}G$ corresponding to the irreducible $\mathbb{Q}$-representation of $G$ affording the character which is the sum of the Galois conjugates of those non-linear irreducible complex characters of $G$ whose kernels contain $Z(G)$. Further, from Lemma \ref{lemma:characterPhi6}, there is a bijection between $\{\bar{\chi} \in \nl(G/K | Z(G)/K) : C_p \cong K < Z(G)\}$ and $\nl(G|Z(G))$, where $\bar{\chi}$ lifts to $\chi \in \nl(G)$. Observe that $(G/K)'\cong C_p \times C_p$. Thus, $G/K$ is a non-abelian $p$-group of order $p^4$ of nilpotency class $3$. Hence, $G/K$ has a unique abelian subgroup $C_{G/K}((G/K)')$ of index $p$ (see \cite[Corollary 55]{Ram}).\\
       Next, fix a subgroup $K$ such that $C_p \cong K < Z(G)$. One can easily see that, there is a bijection between $\nl(G/K | Z(G)/K)$ and $\{\chi \in \nl(G|Z(G)) : K \subseteq \ker(\chi)\}$. Hence, from \cite[Theorem 3]{Ram}, the simple components of the Wedderburn decomposition of $\mathbb{Q}G$ corresponding to all irreducible $\mathbb{Q}$-representations of $G$ that afford the characters $\Omega(\chi)$ for some $\chi \in \nl(G|Z(G))$ such that  $K \subseteq \ker(\chi)$ contribute
	   $$\bigoplus_{\substack{d_K|m_K\\ d_K \nmid m_K'}} \frac{a_{d_K}}{p}M_p(\mathbb{Q}(\zeta_{d_K})) \bigoplus_{\substack{d_K|m_K\\ d_K|m_K'}}\frac{a_{d_K}-a_{d_K}'}{p}M_p(\mathbb{Q}(\zeta_{d_K}))$$
   	   in $\mathbb{Q}G$, where $m_K$ and $m_K'$ are the exponents of $C_{G/K}((G/K)')$ and $C_{G/K}((G/K)')/(Z(G)/K)$, and $a_{d_K}$ and $a_{d_K}'$ are the number of cyclic subgroups of order $d_K$ of $C_{G/K}((G/K)')$ and $C_{G/K}((G/K)')/(Z(G)/K)$, respectively. Therefore, the proof follows.
\end{proof}

 Before describing the Wedderburn decomposition of rational group algebras of the groups of order $p^5$ belonging to the rest of the isoclinic families, we state Reiner's result. 
 
 \begin{lemma}\cite[Theorem 3]{IR} \label{Reiner}
 	Let $\mathbb{K}$ be an arbitrary field with characteristic zero and  $\mathbb{K}^*$ be the algebraic closure of $\mathbb{K}$. Suppose $T$ is an irreducible $\mathbb{K}$-representation of $G$, and extend $T$ (by linearity) to a $\mathbb{K}$-representation of $\mathbb{K}G$. Set
 	\[A=\{T(x) : x\in \mathbb{K}G\}.\]
 	Then $A$ is a simple algebra over $\mathbb{K}$, and we may write  $A= M_{n}(D)$, where $D$ is a division ring. Further,
 	\begin{center}
 		$Z(D)\cong \mathbb{K}(\chi_i)$ and $[D : Z(D)]=(m_\mathbb{K}(\chi_i))^2~(1\leq i \leq k),$
 	\end{center}
 	where $U_i$ are irreducible $K^{*}$-representations of $G$ that affords the character $\chi_i$, $T=m_{\mathbb{K}}(\chi_i)\bigoplus_{i=1}^k U_i$, $k=[\mathbb{K}(\chi_i) : \mathbb{K}]$ and $m_{\mathbb{K}}(\chi_i)$ is the Schur index of $\chi_i$ over $\mathbb{K}$.
 \end{lemma}

Now, we are ready to prove Theorem \ref{thm:rationalgroupalgebraPhi9}, which gives a combinatorial formulation for the Wedderburn decomposition of the rational group algebras of groups of order $p^5$ that belong to $\Phi_9$.
\begin{theorem}\label{thm:rationalgroupalgebraPhi9}
		Let $G$ be a group of order $p^5$ such that $G \in \Phi_9$. Let $m$ and $m'$ denote the exponents of $C_G(G')$ and $C_G(G')/G'$ respectively. Then the Wedderburn decomposition of $\mathbb{Q}G$ is as follows:
	\[\mathbb{Q}G\cong\bigoplus\mathbb{Q}(G/G')\bigoplus_{d|m, d \nmid m'} \frac{a_{d}}{p}M_p(\mathbb{Q}(\zeta_{d}))\bigoplus_{d|m, d|m'}\frac{a_{d}-a_{d}'}{p}M_p(\mathbb{Q}(\zeta_{d})),\]
	where $a_{d}$ and $a_{d}'$ are the number of cyclic subgroups of order $d$ of $C_G(G')$ and $C_G(G')/G'$, respectively.
\end{theorem}
\begin{proof}
	Let $G$ be a group of order $p^5$ such that $G \in \Phi_9$, and let $\chi\in \nl(G)$. Then from Lemma \ref{lemma:characterPhi3Phi9}, there exists $\psi \in \Irr(C_G(G') | G')$ such that $\chi = \psi^G$, and if $\psi \in \Irr(C_G(G') | G')$ then $\psi^G \in \nl(G)$. Now by Theorem \ref{thm:reqpairPhi9}, $\mathbb{Q}(\chi)=\mathbb{Q}(\psi)$. Observe that $\chi^{\sigma}=(\psi^{\sigma})^G$, where $\sigma \in \gal(\mathbb{Q}(\chi)/\mathbb{Q})$ and there are exactly $p$ distinct conjugates $\psi\in \Irr(C_G(G') | G{}')$ such that $\chi=\psi^G$ (see the proof of Lemma \ref{lemma:characterPhi3Phi9}). Let $X$ and $Y$ be the representative set of distinct Galois conjugacy classes of $\Irr(G)$ and $\Irr(C_G(G'))$, respectively. Let $d$ be a divisor of $\exp(C_G(G'))$ such that $\mathbb{Q}(\chi)=\mathbb{Q}(\psi) = \mathbb{Q}(\zeta_d)$. Set $m=\exp(C_G(G'))$ and $m{}'=\exp(C_G(G')/G{}')$. Then we have two cases.\\
	{\bf Case 1 ($d \mid m$ but $d \nmid m'$).} In this case, we have
	\begin{align*}
		|\{\chi \in X : \chi(1)=p, \mathbb{Q}(\chi)=\mathbb{Q}(\zeta_d) \}|&=\frac{1}{p}|\{\psi \in Y : \psi\in \Irr(C_G(G')|G{}'), ~\mathbb{Q}(\psi)=\mathbb{Q}(\zeta_d) \}|\\
		&= \frac{a_d}{p}, 
	\end{align*}
	where $a_d$ denotes the number of cyclic subgroups of order $d$ of $C_G(G')$ (see Lemma \ref{lemma:Ayoub}). \\
	{\bf Case 2 ($d \mid m$ and $d \mid m'$).} In this case, we have
	\begin{align*}
		|\{\chi \in X : \chi(1)=p, \mathbb{Q}(\chi)=\mathbb{Q}(\zeta_d) \}|&=\frac{1}{p}|\{\psi \in Y : \psi\in \Irr(C_G(G')|G{}'), ~\mathbb{Q}(\psi)=\mathbb{Q}(\zeta_d) \}|\\
		&= \frac{a_d-a'_d}{p}, 
	\end{align*}
	where $a_d$ and $a'_d$ denote the number of cyclic subgroups of order $d$ of $C_G(G')$ and $C_G(G')/G{}'$, respectively (see Lemma \ref{lemma:Ayoub}).\\
	Now, let $A_{\mathbb{Q}}(\chi)$ be the simple component of the Wedderburn decomposition of $\mathbb{Q}G$ corresponding to rational representation of $G$ that affords the character $\Omega(\chi)$. Then $A_{\mathbb{Q}}(\chi) \cong M_n(D)$ for some $n \in \mathbb{N}$ and a division ring $D$. From Lemma \ref{lemma:schurindexpgroup}, $m_{\mathbb{Q}}(\chi)=1$. Hence, from Lemma \ref{Reiner}, $D = Z(D) \cong \mathbb{Q}(\chi)$. Furthermore, as $m_{\mathbb{Q}}(\chi)=1$, we get $n=\chi(1)=p$ (see \cite[Theorem 3.3.1]{JR}). Therefore, all the irreducible rational representations of $G$ whose kernel do not contains $G{}'$ will contribute
	$$\bigoplus_{d|m, d \nmid m'} \frac{a_{d}}{p}M_p(\mathbb{Q}(\zeta_{d}))\bigoplus_{d|m, d|m'}\frac{a_{d}-a_{d}'}{p}M_p(\mathbb{Q}(\zeta_{d}))$$
	in the Wedderburn decomposition of $\mathbb{Q}G$. This completes the proof of Theorem \ref{thm:rationalgroupalgebraPhi9}.
\end{proof}

Next, we prove Theorem \ref{thm:rationalgroupalgebraPhi3}, which characterizes the Wedderburn decomposition of the rational group algebras associated with the groups of order $p^5$ that belong to $\Phi_3$.
\begin{theorem}\label{thm:rationalgroupalgebraPhi3}
	Let $G$ be a group of order $p^5$ ($p \geq 5$) such that $G \in \Phi_3$. Then we have the following.
	\begin{enumerate}
		\item For $G \in \Phi_3 \setminus \{\Phi_3(311)b_r~(r=1, \nu), \Phi_3(221)a, \Phi_3(2111)e \}$, let $m$ and $m'$ denote the exponents of $C_G(G')$ and $C_G(G')/G'$ respectively. Then the Wedderburn decomposition of $\mathbb{Q}G$ is as follows:
		$$\mathbb{Q}G\cong\bigoplus\mathbb{Q}(G/G')\bigoplus_{d|m, d \nmid m'} \frac{a_{d}}{p}M_p(\mathbb{Q}(\zeta_{d}))\bigoplus_{d|m, d|m'}\frac{a_{d}-a_{d}'}{p}M_p(\mathbb{Q}(\zeta_{d})),$$
		where $a_{d}$ and $a_{d}'$ are the number of cyclic subgroups of order $d$ of $C_G(G')$ and $C_G(G')/G'$, respectively.
		
\item  For $G=\Phi_3(311)b_r$, where $r=1, \nu$, the Wedderburn decomposition of $\mathbb{Q}G$ is as follows:
		$$\mathbb{Q}G \cong \mathbb{Q} \bigoplus(p+1)\mathbb{Q}(\zeta_p) \bigoplus p\mathbb{Q}(\zeta_{p^2}) \bigoplus pM_p(\mathbb{Q}(\zeta_p)) \bigoplus M_p(\mathbb{Q}(\zeta_{p^3})).$$
		
			\item For $G \in \{\Phi_3(221)a, \Phi_3(2111)e \}$, the Wedderburn decomposition of $\mathbb{Q}G$ is as follows:
		$$\mathbb{Q}G \cong \mathbb{Q} \bigoplus(p+1)\mathbb{Q}(\zeta_p) \bigoplus p\mathbb{Q}(\zeta_{p^2}) \bigoplus 2pM_p(\mathbb{Q}(\zeta_p)) \bigoplus (p-1)M_p(\mathbb{Q}(\zeta_{p^2})).$$
			\end{enumerate}
\end{theorem}
\begin{proof}
	\begin{enumerate}
		\item Using Lemma \ref{lemma:characterPhi3Phi9} and Theorem \ref{thm:reqpairPhi3}, the proof follows from a discussion similar to that in the proof of Theorem \ref{thm:rationalgroupalgebraPhi9}.
	
			\item Let 
		$G=\Phi_3(311)b_r= \langle \alpha, \alpha_1, \alpha_2, \alpha_3 : [\alpha_1, \alpha]=\alpha_2, [\alpha_2, \alpha]^r=\alpha_1^{p^2}=\alpha_3, \alpha^p=\alpha_2^p=\alpha_3^p=1\rangle$
		for $r=1, \nu$ (see \cite[Subsection 4.5]{RJ}). Here, $G'=\langle \alpha_2, \alpha_3 \rangle \cong C_p \times C_p$ and $Z(G)=\langle \alpha_1^p \rangle \cong C_{p^2}$. Since $G/G' \cong C_{p^2} \times C_p$, from Theorem \ref{Perlis-walker}, the simple components of the Wedderburn decomposition of $\mathbb{Q}G$ corresponding to all those irreducible $\mathbb{Q}$-representations of $G$ whose kernels contain $G'$ are
		$$ \mathbb{Q} \bigoplus(p+1)\mathbb{Q}(\zeta_p) \bigoplus p\mathbb{Q}(\zeta_{p^2}).$$
		Further, from Theorem \ref{thm:reqpairPhi3}(3)(a), $|\{\chi \in\nl(G) : \mathbb{Q}(\chi) = \mathbb{Q}(\zeta_{p^3})\}|=p^3-p^2$. Hence, these irreducible complex characters constitute a single Galois conjugacy class, and the corresponding simple component of the Wedderburn decomposition of $\mathbb{Q}G$ is
		$$M_p(\mathbb{Q}(\zeta_{p^3})).$$
		Furthermore, from Theorem \ref{thm:reqpairPhi3}(3)(b), there is a bijection between $\nl(G/K)$, where $K=\langle \alpha_3=\alpha_1^{p^2} \rangle <Z(G)$ and $\{\chi \in\nl(G) : \mathbb{Q}(\chi) \neq \mathbb{Q}(\zeta_{p^3})\}$. Observe that $G/K\cong \Phi_2(211)c$ is a VZ $p$-group of order $p^4$. Therefore, from Theorem \ref{lemma:Wedderburn VZ}, the corresponding simple component of the Wedderburn decomposition of $\mathbb{Q}G$ is
		$$\bigoplus pM_p(\mathbb{Q}(\zeta_p)).$$
		This completes the proof of Theorem \ref{thm:rationalgroupalgebraPhi3}(2).
		
		\item For $G=\Phi_3(221)a= \langle \alpha, \alpha_1, \alpha_2, \alpha_3 : [\alpha_1, \alpha]=\alpha_2, [\alpha_2, \alpha]=\alpha^p=\alpha_3, \alpha_1^{p^2}=\alpha_2^p=\alpha_3^p=1\rangle$, we have $G'=\langle \alpha_2, \alpha_3 \rangle \cong C_p \times C_p$ and $Z(G)=\langle \alpha_1^p, \alpha_3 \rangle \cong C_p \times C_p$. For $G=\Phi_3(2111)e= \langle \alpha, \alpha_1, \alpha_2, \alpha_3 : [\alpha_1, \alpha]=\alpha_2, [\alpha_2, \alpha]=\alpha_3, \alpha^p=\alpha_1^{p^2}=\alpha_2^p=\alpha_3^p=1\rangle$, we have $G'=\langle \alpha_2, \alpha_3 \rangle \cong C_p \times C_p$ and $Z(G)=\langle \alpha_1^p, \alpha_3 \rangle \cong C_p \times C_p$. One can easily see that for $G \in \{\Phi_3(221)a, \Phi_3(2111)e \}$, we have  $G/G' \cong C_{p^2} \times C_p$ and $G/Z(G) \cong \Phi_2(1^3)$ (see \cite[Subsection 4.1]{RJ}). Therefore, from Theorem \ref{Perlis-walker}, the simple components of the Wedderburn decomposition of $\mathbb{Q}G$ corresponding to all those irreducible $\mathbb{Q}$-representations of $G$ whose kernels contain $G'$ are
		$$ \mathbb{Q} \bigoplus(p+1)\mathbb{Q}(\zeta_p) \bigoplus p\mathbb{Q}(\zeta_{p^2}).$$
		Next, for $G \in \{\Phi_3(221)a, \Phi_3(2111)e \}$, there is a bijection between $\{\bar{\chi} \in \nl(G/K | Z(G)/K) : C_p \cong K < Z(G)\}$ and $\nl(G|Z(G))$, where $\bar{\chi}$ lifts to $\chi \in \nl(G)$ (see Theorem \ref{thm:reqpairPhi3}(2)). Note that for each proper subgroup $K$ of $Z(G)$, the quotient $G/K$ is a non-abelian $p$-group of order $p^4$. By employing an argument analogous to the one used in the proof of Theorem \ref{thm:rationalgroupalgebraPhi3}(2), together with Theorem \ref{lemma:Wedderburn VZ} and \cite[Theorem 3]{Ram}, we can compute the simple components of the Wedderburn decomposition of $\mathbb{Q}G$ corresponding to all those irreducible $\mathbb{Q}$-representations of $G$ whose kernels do not contain $Z(G)$. These components are given by
		$$\bigoplus (2p-1)M_p(\mathbb{Q}(\zeta_p)) \bigoplus (p-1)M_p(\mathbb{Q}(\zeta_{p^2})).$$
		Furthermore, all non-linear irreducible complex characters of $G$ whose kernels contain $Z(G)$ form a single Galois conjugacy class. The corresponding simple component of the Wedderburn decomposition of $\mathbb{Q}G$ is
		$$M_p(\mathbb{Q}(\zeta_{p})).$$
		By combining all these simple components of the Wedderburn decomposition of $\mathbb{Q}G$, we obtain the desired result, thereby completing the proof of Theorem \ref{thm:rationalgroupalgebraPhi3}(3). \qedhere
	\end{enumerate} 
\end{proof}

Theorem \ref{thm:rationalgroupalgebraPhi7Phi8} characterizes the Wedderburn decomposition of rational group algebras of all groups of order $p^5$ of $\Phi_7$ and $\Phi_8$.
\begin{theorem}\label{thm:rationalgroupalgebraPhi7Phi8}
	Let $G$ be a group of order $p^5$ such that $G \in \Phi_7 \cup \Phi_8$. Then the Wedderburn decomposition of $\mathbb{Q}G$ is as follows:
	\[\mathbb{Q}G \cong \bigoplus\mathbb{Q}(G/G') \bigoplus p M_p(\mathbb{Q}(\zeta_p)) \bigoplus M_{p^2}(\mathbb{Q}(\zeta_p)).\]
\end{theorem}
\begin{proof}
	Let $G$ be a group of order $p^5$ such that $G \in \Phi_7 \cup \Phi_8$. Then $\cd(G)=\{1, p, p^2\}$ and each non-linear irreducible complex character of degree $p$ of $G$ is the lift of some non-linear irreducible complex character of $G/Z(G)$ (see Lemma \ref{lemma:characterPhi7Phi8Phi10}). Further, if $G \in \Phi_{7}$ then $G/Z(G) \cong \Phi_2(1^4)$ (see \cite[Subsection 4.1]{RJ}), which is a VZ $p$-group of order $p^4$. Observe that the center of $\Phi_2(1^4)$ is isomorphic to $C_p \times C_p$. Similarly, if $G \in \Phi_{8}$ then $G/Z(G) \cong \Phi_2(22)$ (see \cite[Subsection 4.1]{RJ}), which is again a VZ $p$-group of order $p^4$. Observe that the center of $\Phi_2(22)$ is again isomorphic to $C_p \times C_p$. Hence, from Theorem \ref{lemma:Wedderburn VZ}, the simple components of the Wedderburn decomposition of $\mathbb{Q}G$ corresponding to all the irreducible $\mathbb{Q}$-representations of $G$ whose characters are sum of Galois conjugates of a degree $p$ complex irreducible character contribute  
	$$\bigoplus p M_p(\mathbb{Q}(\zeta_p))$$
	in the Wedderburn decomposition of $\mathbb{Q}G$. Furthermore, the pair $(G, Z(G))$ is a Camina pair and there is a bijection between the sets $\Irr^{(p^2)}(G)$ and $\Irr(Z(G)) \setminus \{1_{Z(G)}\}$ (see Lemma \ref{lemma:characterPhi7Phi8Phi10}). Since $Z(G) \cong C_p$, from Lemma \ref{Reiner}, 
	\[\bigoplus M_{p^2}(\mathbb{Q}(\zeta_p))\]
	is the only simple component of the Wedderburn decomposition of $\mathbb{Q}G$ corresponding to the irreducible $\mathbb{Q}$-representation of $G$ whose character is the sum of Galois conjugates of a degree $p^2$ complex irreducible character. This completes the proof of Theorem \ref{thm:rationalgroupalgebraPhi7Phi8}.
\end{proof}

 Theorem \ref{thm:rationalgroupalgebraPhi10} gives combinatorial formulations for the Wedderburn decomposition of rational group algebras of all $p$-groups of order $p^5$ of $\Phi_{10}$.

\begin{theorem}\label{thm:rationalgroupalgebraPhi10}
	Let $G$ be a group of order $p^5$ such that $G \in \Phi_{10}$.  Then the Wedderburn decomposition of $\mathbb{Q}G$ is as follows:
	\begin{enumerate}
		\item for $p \geq 5$, \[\mathbb{Q}G \cong \bigoplus\mathbb{Q}(G/G') \bigoplus (p+1) M_p(\mathbb{Q}(\zeta_p)) \bigoplus M_{p^2}(\mathbb{Q}(\zeta_p)), ~ \text{and}\]
		\item for $p =3 $, \[\mathbb{Q}G \cong \bigoplus\mathbb{Q}(G/G') \bigoplus M_3(\mathbb{Q}(\zeta_3)) \bigoplus M_3(\mathbb{Q}(\zeta_9)) \bigoplus M_9(\mathbb{Q}(\zeta_3)).\]
	\end{enumerate}
\end{theorem}
\begin{proof}
	Let $G$ be a group of order $p^5$ such that $G \in \Phi_{10}$. Then $\cd(G)=\{1, p, p^2\}$ and each non-linear irreducible complex character of degree $p$ of $G$ is the lift of some non-linear irreducible complex character of $G/Z(G)$ (see Lemma \ref{lemma:characterPhi7Phi8Phi10}). Further, $G/Z(G) \cong \Phi_3(1^4)$ (see \cite[Subsection 4.1]{RJ}). Thus, from \cite[Corollary 58]{Ram}, the simple components of the Wedderburn decomposition of $\mathbb{Q}G$ corresponding to all the irreducible $\mathbb{Q}$-representations of $G$ whose characters are sum of Galois conjugates of a degree $p$ complex irreducible character contribute  
	\[\bigoplus (p+1)M_p(\mathbb{Q}(\zeta_p))\]
	in $\mathbb{Q}G$ for $p \geq 5$. For $p=3$, the simple components of the Wedderburn decomposition of $\mathbb{Q}G$ corresponding to all the irreducible $\mathbb{Q}$-representations of $G$ whose characters are sum of Galois conjugates of a degree $3$ complex irreducible character contribute  
	\[\bigoplus M_3(\mathbb{Q}(\zeta_3)) \bigoplus M_3(\mathbb{Q}(\zeta_9))\]
	in $\mathbb{Q}G$ (see \cite[Remark 59]{Ram}). Furthermore, the pair $(G, Z(G))$ is a Camina pair and there is a bijection between the sets $\Irr^{(3^2)}(G)$ and $\Irr(Z(G)) \setminus \{1_{Z(G)}\}$ (see Lemma \ref{lemma:characterPhi7Phi8Phi10}). Since $Z(G) \cong C_3$, from Lemma \ref{Reiner}, 
	\[\bigoplus M_{3^2}(\mathbb{Q}(\zeta_3))\]
	is the only simple component of the Wedderburn decomposition of $\mathbb{Q}G$ corresponding to the irreducible $\mathbb{Q}$-representation of $G$ whose character is the sum of Galois conjugates of a degree $3^2$ complex irreducible character. This completes the proof of Theorem \ref{thm:rationalgroupalgebraPhi10}.
\end{proof}
Finally, we conclude this section by presenting Corollary \ref{cor:isogpalgPhi_7,8,10}.
\begin{corollary}\label{cor:isogpalgPhi_7,8,10}
	Let $G$ and $H$ be two isoclinic groups of order $p^5$ such that $G \in \Phi_{7} \cup \Phi_{8} \cup \Phi_{10}$. Then $\mathbb{Q}G \cong \mathbb{Q}H$ if and only if $G/G' \cong H/H'$.
\end{corollary} \begin{proof}
It follows from Theorems \ref{thm:rationalgroupalgebraPhi7Phi8} and \ref{thm:rationalgroupalgebraPhi10}.
\end{proof}

\section{Computational verification} \label{sec:ComputationalVerifiaction} 

All inequivalent irreducible representations of a small order pc-group over an arbitrary field can be constructed using the intrinsics \texttt{IrreducibleModulesSchur} and \texttt{IrreducibleRepresentationsSchur} in \textsc{Magma}. Constructing irreducible rational matrix representations using required pairs provides an alternative approach, particularly for manual computations when a required pair is known. To verify the computational results presented in Section \ref{section:rational representation}, we provide a \textsc{Magma} implementation in the \texttt{GitHub} repository \cite{Repository}. The script \texttt{RationalRepsReqPairs.m} facilitates the reproduction and validation of our findings related to rational representations. It includes functions to construct irreducible rational matrix representations, verify required pairs, and compute all inequivalent irreducible rational representations.

	\section{Acknowledgments}
	The first author acknowledges the financial support of the University Grants Commission, Government of India. We thank E.A. O'Brien for helping us to write the \textsc{Magma} code to computationally verify several results of this article, carefully reviewing multiple versions of the paper, and engaging in detailed discussions on its content. We also acknowledge the University of Auckland for providing remote access to its computational facilities.


\begin{thebibliography}{99}
		\bibitem{Ayoub} R. G. Ayoub and C. Ayoub, On the group rings of a finite abelian group, Bull. Austral. Math. Soc. {\bf 1} (1969) 245-261.
		\bibitem{BM14} G. K. Bakshi and S. Maheshwary, The rational group algebra of a normally monomial group, J. Pure Appl. Algebra {\bf 218}(9) (2014), 1583–1593.
		\bibitem{BGO} G.~K. Bakshi, J. Garg and G. Olteanu, Rational group algebras of generalized strongly monomial groups: primitive idempotents and units, Math. Comp. {\bf 93}(350) (2024), 3027--3058.
		\bibitem{Berkovich} Y. Berkovich, Groups of prime power order, Volume 1, (De Gruyter, Berlin, New York, 2008).
		\bibitem{Bouc} Serge Bouc, The functor of rational representations for $p$-groups, Adv. Math. {\bf 186}(2) (2004), 267–306.
		\bibitem{Magma} W. Bosma, J.~J. Cannon and C. Playoust, The Magma algebra system. I. The user language, J. Symbolic Comput. {\bf 24}(3-4) (1997), 235--265.
		\bibitem{Burnside} W. Burnside, On the arithmetical nature of the coefficients in a group of linear substitutions, Proc. Lond. Math. Soc. {\bf 7} (1908), 8-13.
		\bibitem{Camina} A. R. Camina, Some conditions which almost characterize Frobenius groups, Israel J. Math. {\bf 31} (1978), 153-160.
		\bibitem{Ram3} R. K. Choudhary and S. K. Prajapati, A combinatorial formula for the Wedderburn decomposition of rational group algebras and the rational representations of ordinary metacyclic $p$-groups, \href{https://arxiv.org/abs/2410.20933}{arXiv:2410.20933 [math.RT]}
		\bibitem{Ram2} R. K. Choudhary and S. K. Prajapati, A combinatorial formula for the Wedderburn decomposition of rational group algebras of split metacyclic $p$-groups, J. Algebra Appl., \href{https://doi.org/10.1142/S0219498826500684}{doi.org/10.1142/S0219498826500684}
		\bibitem{Ram} R. K. Choudhary and S. K. Prajapati, Rational representations and rational group algebra of VZ $p$-groups, J. Aust. Math. Soc., {\bf 118}(1) (2025), 1-30.
		\bibitem{Repository} R. K. Choudhary and S. K. Prajapati, On Matrix Representations of Groups of Order $p^5$ over $\mathbb{Q}$, \href{https://github.com/RamKaranChoudhary/On-Matrix-Representations-of-Groups-of-Order-p5-over-Q}{github.com/RamKaranChoudhary/On-Matrix-Representations-of-Groups-of-Order-p5-over-Q}
		\bibitem{Dabbaghian} V. Dabbaghian-Abdoly, An algorithm for constructing representations of finite groups, J. Symbolic Comput. {\bf 39}(6) (2005), 671-688.
		\bibitem{Dixon} V. Dabbaghian and J. D. Dixon, Computing matrix representations, Math. Comp. {\bf 79}(271) (2010), 1801–1810.
		\bibitem{FM} G. A. Fernández-Alcober and A. Moretó, Groups with two extreme character degrees and their normal subgroups, Trans. Amer. Math. Soc. {\bf 353}(6) (2001), 2171-2192.
		\bibitem{Grittini}  N. Grittini, Odd degree rational characters and the order of rational elements in finite groups, J. Algebra {\bf 654} (2024), 59–69.
		\bibitem{PH} P. Hall, The classification of prime-power groups, J. Reine Angew. Math. {\bf 182} (1940), 130-141.
		\bibitem{Herman} A. Herman, On the automorphism group of rational group algebras of metacyclic groups, Comm. Algebra {\bf 25}(7) (1997), 2085–2097.
		\bibitem{I} I. M. Isaacs, Character Theory of Finite Groups (Dover Publications Inc., New York, 1994).
		\bibitem{RJ} R. James, The groups of order $p^6$ ($p$ an odd prime), Math. Comp. {\bf 34}(150) (1980), 613-637.
		\bibitem{Jes-Rio} E. Jespers and Á. del Río, A structure theorem for the unit group of the integral group ring of some finite groups, J. Reine Angew. Math. {\bf 521} (2000), 99–117.
		\bibitem{Jes-Lea-Paq} E. Jespers, G. Leal and A. Paques, Central idempotents in rational group algebras of finite nilpotent groups, J. Algebra Appl. {\bf 2}(1) (2003), 57–62.
        \bibitem{Jes-Olt-Rio} E. Jespers, G. Olteanu and \'A. del~R\'io, Rational group algebras of finite groups: from idempotents to units of integral group rings, Algebr. Represent. Theory {\bf 15}(2) (2012), 359--377.
        \bibitem{JR} E. Jespers and Á. del Río, Group Ring Groups, Volume 1: Orders and Generic Constructions of Units (De Gruyter, Berlin, 2016).
		\bibitem{Kletzing} D. Kletzing, Rational representations of finite groups: the story of $\gamma$, Amer. Math. Monthly {\bf 94}(9) (1987), 846–854.
		\bibitem{MLL} M. L. Lewis, The vanishing-off subgroup, J. Algebra {\bf 321}(4) (2009), 1313-1325.
		\bibitem{MLL2} M. L. Lewis, Generalizing Camina groups and their character tables, J. Group Theory, {\bf 12}(2) (2009), 209-218.
		\bibitem{MLL3} M. L. Lewis, On $p$-group Camina pairs, J. Group Theory {\bf 15}(2012), 469-483.
		\bibitem{Navarro} G. Navarro and L. Sanus, Rationality and normal 2-complements, J. Algebra {\bf 320}(6) (2008), 2451–2454.
		\bibitem{Navarro1} G. Navarro and P. H. Tiep, Rational irreducible characters and rational conjugacy classes in finite groups, Trans. Amer. Math. Soc. {\bf 360}(5) (2008), 2443–2465.
		\bibitem{Navarro2} G. Navarro and P. H. Tiep, Degrees of rational characters of finite groups, Adv. Math. {\bf 224}(3) (2010), 1121–1142.
		\bibitem{Olt07} G. Olteanu, Computing the Wedderburn decomposition of group algebras by the Brauer-Witt theorem, Math. Comp. {\bf 76}(258) (2007), 1073–1087.
		\bibitem{PW} S. Perlis and G. L. Walker, Abelian group algebras of finite order, Trans. Amer. Math. Soc., {\bf 68}(3) (1950), 420-426.
		\bibitem{Plesken}  W. Plesken and B. Souvignier, Constructing rational representations of finite groups, Experiment. Math. {\bf 5}(1) (1996), 39–47.
		\bibitem{SKP} S. K. Prajapati, M. R. Darafsheh and M. Ghorbani, Irreducible characters of $p$-group of order $\leq p^5$, Algebras and Representation Theory {\bf 20} (2017), 1289-1303.
		\bibitem{SKP1} S. K. Prajapati and R. Sharma, Total character of a group $G$ with ($G$, $Z(G)$) as a generalized Camina pair, Canad. Math. Bull. {\bf 59}(2) (2016), 392-402.
		\bibitem{IR} I. Reiner, The Schur index in the theory of group representations, Michigan Mat. J., {\bf 8} (1961), 39-47.
		\bibitem{Rit-Seh} J. Ritter and S. K. Sehgal, Construction of units in integral group rings of finite nilpotent groups, Trans. Amer. Math. Soc. {\bf 324}(2) (1991), 603–621.	
		\bibitem{Serre}J. P. Serre, Linear Representations of Finite Groups, Translated from the French by L. L. Scott (Springer Verlag New York, Inc., 1977).
		\bibitem{Yam} T. Yamada, The Schur Subgroup of the Brauer Group, Lecture Notes in Mathematics, 397 (Springer Verlag, Berlin, 1974).
		\bibitem{Y} T. Yamada, Remarks on rational representations of a finite group, SUT J. Math. {\bf 29}(1) (1993), 71-77.	
		\bibitem{Gap} The GAP Group, GAP– Groups, Algorithms, and Programming, Version 4.13.0, 2024, \href{https://www.gap-system.org}{https://www.gap-system.org}
	\end{thebibliography}
\end{document}